\documentclass[reqno,11pt]{amsart}
\usepackage{amsfonts}
\usepackage{amssymb}
\usepackage{mathrsfs}
\usepackage{mathtools}
\usepackage[colorlinks=true,linkcolor=blue,citecolor=blue,urlcolor=blue]{hyperref}

\newtheorem{theorem}{Theorem}[section]
\newtheorem{corollary}[theorem]{Corollary}
\newtheorem{lemma}[theorem]{Lemma}
\newtheorem{proposition}[theorem]{Proposition}
\newtheorem{conjecture}[theorem]{Conjecture}
\theoremstyle{definition}
\newtheorem{definition}[theorem]{Definition}

\numberwithin{equation}{section}

\newcommand{\RR}{\mathbb{R}}
\newcommand{\CC}{\mathbb{C}}
\newcommand{\QQ}{\mathbb{Q}}
\newcommand{\ZZ}{\mathbb{Z}}
\newcommand{\NN}{\mathbb{N}}
\newcommand{\be}{\mathbf{e}}
\newcommand{\bff}{\mathbf{f}}

\newcommand{\bStab}{\mathbf{Stab}}
\newcommand{\bG}{\mathbf{G}}
\newcommand{\bS}{\mathbf{S}}
\newcommand{\cA}{\mathcal{A}}
\newcommand{\cB}{\mathcal{B}}
\newcommand{\cC}{\mathcal{C}}
\newcommand{\cD}{\mathcal{D}}

\newcommand{\cK}{\mathcal{K}}
\newcommand{\cL}{\mathcal{L}}
\newcommand{\cP}{\mathcal{P}}
\newcommand{\cO}{\mathcal{O}}
\newcommand{\cS}{\mathcal{S}}
\newcommand{\cW}{\mathcal{W}}
\newcommand{\cU}{\mathcal{U}}
\newcommand{\SL}{\mathrm{SL}}
\newcommand{\Ad}{\mathrm{Ad}}
\newcommand{\SO}{\mathrm{SO}}
\newcommand{\ad}{\mathrm{ad}}
\newcommand{\Rep}{\mathrm{Re}}
\newcommand{\Lie}{\mathrm{Lie}}
\newcommand{\Gal}{\mathrm{Gal}}
\newcommand{\covol}{\mathrm{covol}}
\newcommand{\coeff}{\mathrm{coeff}}
\newcommand{\op}{\mathrm{op}}
\newcommand{\diag}{\mathrm{diag}}
\newcommand{\dist}{\mathrm{dist}}
\newcommand{\supp}{\mathrm{supp}}

\newcommand{\Bad}{\mathbf{Bad}}
\newcommand{\tr}{\mathrm{tr}}
\newcommand{\id}{\mathrm{id}}
\newcommand{\fsl}{\mathfrak{sl}}
\newcommand{\fa}{\mathfrak{a}}
\newcommand{\fg}{\mathfrak{g}}
\newcommand{\fk}{\mathfrak{k}}
\newcommand{\fh}{\mathfrak{h}}
\newcommand{\fn}{\mathfrak{n}}
\newcommand{\fp}{\mathfrak{p}}
\newcommand{\la}{\left\langle}
\newcommand{\ra}{\right\rangle}
\renewcommand{\emptyset}{\varnothing}
\renewcommand{\setminus}{\smallsetminus}

\begin{document}

\title[Bounded orbits on homogeneous spaces]{Bounded orbits on homogeneous spaces and\\ bounded geodesics on locally symmetric spaces}
\date{}
\author{Chengyang Wu}
\address{Department of Mathematics, University of Chicago, Chicago, Illinois, 60637, U.S.
}
\email{chengyangwu1999@gmail.com}

\begin{abstract}
We present a unified approach to proving the following two results:
\begin{enumerate}
    \item[(i)] Let $G$ be a Lie group, $\Gamma$ a lattice in $G$, $X=G/\Gamma$, and let $F$ be a one-parameter $\Ad$-diagonalizable subgroup of $G$. Then the set $$
\{x\in X:Fx\text{ is bounded in }X\}
$$
is hyperplane absolute winning on $X$. This fully solves a conjecture \cite[Conjecture 7.1]{AGK} of An-Guan-Kleinbock in 2015.
    \item[(ii)] Let $G$ be a connected semisimple Lie group with finite center and without compact factors, $K\leq G$ a maximal compact subgroup, $\Gamma\leq G$ a torsion-free lattice, and $M=K\backslash G/\Gamma$. Then for any $p\in M$, the set $$
    \{\xi\in S_pM:\gamma(p,\xi)\text{ is bounded in }M\}
    $$
    is hyperplane absolute winning on $S_pM$, where $\gamma(p,\xi)$ denotes the geodesic ray through the base point $p$ in the direction $\xi$. This fully solves a conjecture \cite[Conjecture 1.2]{GW26} of Guan-Wu in 2026.
\end{enumerate}
\end{abstract}

\maketitle

\section{Introduction}\label{S:introduction}

\subsection{Bounded orbits on homogeneous spaces}\label{SS:statement}
Let $G$ be a Lie group, $\Gamma$ a lattice in $G$, and $F=\{g_t:t\in\RR\}$ a one-parameter subgroup of $G$. The action of $F$ on the homogeneous space $X=G/\Gamma$ by left translations defines a flow. Consider the set \begin{equation}
    E_X(F):=\{x\in X:Fx\text{ is bounded in }X\}.
\end{equation}
This set carries important information about the dynamical system $(X,F)$ and has been extensively studied. Clearly, when $X$ is non-compact and $\mu$ is an $F$-ergodic measure on $X$ with full support, one has $\mu(E_X(F))=0$. On the other hand, for certain classes of one-parameter subgroups $F$ of $G$, these sets $E_X(F)$ can be shown to be quite substantial--in particular, they are thick. Here and hereafter, we say that a subset is \textsl{thick} if its intersection with any nonempty open subset has full Hausdorff dimension in the ambient space. 

The main focus of this paper is the case when $F=\{g_t:t\in\RR\}$ is $\Ad$-diagonalizable. Throughout this paper, we say that $F$ is \textsl{$\Ad$-diagonalizable} if $\Ad(g_t)$ is diagonalizable over $\CC$ for every $t\in\RR$, and \textsl{$\Ad$-$\RR$-diagonalizable} if it is diagonalizable over $\RR$. In 1996, a seminal work \cite{KM96} by Kleinbock and Margulis implies that: if $F$ is $\Ad$-diagonalizable, then the set $E_X(F)$ is thick in $X$. However, for two arbitrary $\Ad$-diagonalizable one-parameter subgroups $F_1\neq F_2$ of $G$, their method cannot even show that \begin{equation}\label{E:nonempty}
    E_X(F_1)\cap E_X(F_2)\neq\varnothing.
\end{equation}
The same limitation also occurred in the later work \cite{KW13} of Kleinbock and Weiss. The authors proved certain modified winning property for the set $E_X(F)$, which, however, cannot be applied to show that winning sets of different games defined by different flows have nonempty intersection.

In history, a general framework has emerged to cope with the above difficulty. Schmidt \cite{Sch66} developed a game played on any complete metric space, whose winning sets satisfy the following two properties: \begin{itemize}
        \item[(W1)] winning subsets of Riemannian manifolds are thick;
        \item[(W2)] a countable intersection of winning sets is again winning.
    \end{itemize}
One may refer to Section \ref{SS:Schmidt} for relevant definitions and properties. In view of (W1) and (W2), if one proves that the set $E_X(F)$ is winning for an arbitrary $\Ad$-diagonalizable one-parameter subgroup $F\subseteq G$, then it is clear that \eqref{E:nonempty} holds.

The very first winning result was established by Schmidt \cite{Sch66,Sch69} in the field of Diophantine approximation. He proved that for any two natural numbers $m$ and $n$, the set of badly approximable $m\times n$ matrices is winning on $\RR^{m\times n}$. In 1980s, Dani \cite{Dani85} established a correspondence between Diophantine properties of matrices and dynamical behaviors of trajectories, and hence showed that \begin{equation}\label{E:Sch-winning}
    \text{the set }E_X(F) \text{ is winning on }X \text{ with respect to Schmidt's game}
\end{equation}
when $X=\SL_d(\RR)/\SL_d(\ZZ)\;(d=m+n)$ and $F\subseteq \SL_d(\RR)$ is a one-parameter diagonal subgroup with two different eigenvalues. Dani \cite{Dani86} also proved the statement \eqref{E:Sch-winning} when $X$ is a finite volume quotient of a connected semisimple Lie group of $\RR$-rank $1$.

Later, more variants of Schmidt's game were introduced, including the hyperplane absolute game and Cantor potential game (see Section \ref{S:pre} for details). Using these tools, people have proved the winning property of $E_X(F)$ in the following cases: \begin{enumerate}
    \item[(i)] (An--Guan--Kleinbock \cite{AGK}) $X=\SL_3(\RR)/\SL_3(\ZZ)$;
    \item[(ii)]  (Guan--Wu \cite{GW18}) $X=\SL_{d}(\RR)/\SL_{d}(\ZZ)$ and $F=\{g_t:t\in\RR\}$ is a one-parameter diagonal subgroup of the form $$g_t=\diag(e^{\lambda t}I_{d-2},e^{\mu t},e^{-t}),\quad \text{where}\quad \lambda>\mu>0\text{ and }(d-2)\lambda+\mu=1.
    $$
    \item[(iii)] (Einsiedler--Ghosh--Lytle \cite{EGL16}, Kleinbock-Ly \cite{KL16}, An--Ghosh--Guan--Ly \cite{AGGL})\\
    $X=\SL_2(\RR)^r/\Gamma$ where $\Gamma$ is a lattice in $\SL_2(\RR)^r$;
    \item[(iv)] (Sawant \cite{Saw23}) $X=\SL_2(\CC)^r/\Gamma$ where $\Gamma$ is a lattice in $\SL_2(\CC)^r$;
    \item[(v)] (Sawant \cite{Saw25}) $X=\SL_3(\CC)/\SL_3(\cO_K)$ where $K$ is an imaginary quadratic field.
\end{enumerate}
In \cite{AGK}, the authors proposed a general conjecture (see \cite[Conjecture 7.1]{AGK}): 

\begin{conjecture}\label{C:AGK}
Let $G$ be a Lie group, $\Gamma$ a lattice in $G$, $X=G/\Gamma$, and let $F$ be a one-parameter $\Ad$-diagonalizable subgroup of $G$. Then the set $E_X(F)$ is HAW on $X$.
\end{conjecture}

Here HAW stands for \textsl{hyperplane absolute winning}, that is, the winning property in the hyperplane absolute game. This paper solves Conjecture \ref{C:AGK} in full generality.

\begin{theorem}\label{T:main-homo}
    Conjecture \ref{C:AGK} holds.
\end{theorem}

In particular, in view of the similar properties of HAW sets as (W1)--(W2) (Lemma \ref{L:HAW-Ra}), the following corollary is immediate:

\begin{corollary}
    Let $G,\Gamma,X$ be as in Conjecture \ref{C:AGK}, and let $(F_n)_{n\in\NN}$ be a sequence of one-parameter $\Ad$-diagonalizable subgroups of $G$. Then the intersection $$\bigcap_{n\in\NN}E_X(F_n)
    =\{x\in X:F_nx\text{ is bounded for all }n\in\NN
    \}$$ is HAW on $X$, and hence is thick in $X$.
\end{corollary}

One may also formulate an equivalent version of Conjecture \ref{C:AGK} on the expanding horospherical subgroup with respect to a chosen direction of the one-parameter $\Ad$-diagonalizable subgroup $F$. 

 Write
\begin{equation}\label{E:flow-horosphere}
    F^+:=\{g_t:t\geq0\},\qquad
    H(F^+):=\{h\in G:g_{-t}hg_t\to e\text{ as }t\to+\infty\}.
\end{equation}
The group $H=H(F^+)$ is the \textsl{expanding horospherical subgroup} associated with $F^+$. For any $x\in X$, we consider the set
\begin{equation}\label{E:bounded-sets}
    E_{H,x}(F^+):=\{h\in H:hx\in E_X(F^+)\}.
\end{equation}
Then we will deduce Theorem \ref{T:main-homo} from the following conjectured result (see \cite[Conjecture 7.2]{AGK}):

\begin{theorem}\label{T:main-H}
    Let $G,\Gamma,X$ be as in Conjecture \ref{C:AGK}, $F^+$ a one-parameter $\Ad$-diagonalizable subsemigroup of $G$, and let $H=H(F^+)$. Then for any $x\in X$, the set $E_{H,x}(F^+)$ is HAW on $H$.
\end{theorem}

In particular, the following corollary is immediate:

\begin{corollary}\label{C:intersection}
    Let $G,\Gamma,X$ be as in Conjecture \ref{C:AGK}, and let $(F^+_n)_{n\in\NN}$ be a sequence of one-parameter $\Ad$-diagonalizable subsemigroup of $G$ with a common expanding horospherical subgroup $H=H(F^+_n)$. Then for any $x\in X$, the set $$
    \bigcap_{n\in\NN}E_{H,x}(F_n^+)=\{
    h\in H:F_n^+hx\text{ is bounded in }X \text{ for all }n\in\NN
    \}
    $$
    is HAW on $H$, and hence is thick in $H$.
\end{corollary}

\subsection{Bounded geodesics on locally symmetric spaces}\label{SS:geo} A connected complete Riemannian manifold $M$ is \textsl{locally symmetric}, if for every $p\in M$, its geodesic symmetry (which fixes the point $p$ and reverses geodesics through $p$) is a local isometry. It is \textsl{of non-compact type} if the isometry group of its universal cover is a semisimple Lie group without compact factors.

Let $M$ be a locally symmetric space of non-compact type with finite Riemannian volume. For any $p\in M$, let $S_p(M)$ denote the unit sphere centered at the origin in the tangent space at $p$. For any $p\in M$ and any $\xi\in S_p(M)$, let $\gamma(p,\xi)$ denote the geodesic ray through the base point $p$ in the direction $\xi$. Of particular interest is the following subset of $S_p(M)$:
	$$
	E^+(p,\infty):=\{\xi\in S_p(M): \gamma(p,\xi) \text{ is bounded in } M\}.
	$$
Following \cite[Corollary 5.5]{KM96}, Guan and the author \cite{GW26} showed that for any $p\in M$, the set $E^+(p,\infty)$ is thick in $S_p(M)$. They also raised the following stronger conjecture (see \cite[Conjecture 1.2]{GW26}):

\begin{conjecture}\label{C:geo}
    Let $M$ be a locally symmetric space of non-compact type with finite Riemannian volume. Then for any $p\in M$, the set $E^+(p,\infty)$ is HAW on $S_pM$.
\end{conjecture}

Conjecture \ref{C:geo} was only proved in some special cases: \begin{enumerate}
    \item[(i)] (Dani \cite{Dani86}, McMullen \cite{McM10}) $M=\SO(n)\backslash \SO^0(n,1)/\Gamma$, where $\Gamma$ is a torsion-free nonuniform lattice.
    \item[(ii)] (Guan--Wu \cite{GW26}) $M=\SO(3)\backslash \SL_3(\RR)/\Gamma$, where $\Gamma$ is a torsion-free subgroup commensurable with $\SL_3(\ZZ)$.
\end{enumerate}
This paper solves Conjecture \ref{C:geo} in full generality.

\begin{theorem}\label{T:main-bi}
    Conjecture \ref{C:geo} holds.
\end{theorem}

Now we briefly review some basic facts concerning locally symmetric spaces (see \cite[Chapter II]{BM00} for details). Let $M$ be a locally symmetric space of non-compact type, and let $\widetilde{M}$ denote its universal cover. The isometry group $G$ of $\widetilde{M}$ has finitely many connected components, whose identity component is a connected semisimple Lie group without compact factors and with trivial center. Fix a point $p_0\in M$, and let $\widetilde{p_0}\in\widetilde{M}$ be a preimage of $p_0$. The stabilizer $K:=\mathrm{Stab}_G(\widetilde{p_0})$ is a maximal compact subgroup of $G$. We identify the globally symmetric space $\widetilde{M}$ with $K\backslash G$, and view the fundamental group $\Gamma:=\pi_1(M)$ as a torsion-free discrete subgroup of $G$ via deck transformations. Then $M$ can be identified with $K\backslash G/\Gamma$. 

Conversely, given any semisimple Lie group $G$ without compact factors and with finite center, any maximal compact subgroup $K\leq G$ and any torsion-free discrete subgroup $\Gamma\leq G$, there exists a unique locally symmetric space structure on $K\backslash G/\Gamma$. Moreover, $M$ has finite Riemannian volume if and only if $\Gamma$ is a lattice in $G$, and the rank of $M$ is the rank of $G$.

Write $\fg=\Lie(G)$ and $\fk=\Lie(K)$. The Cartan decomposition gives an $\Ad(K)$-invariant complementary subspace $\fp\subseteq\fg$ to the subalgebra $\fk$, and a diffeomorphism $$
	K\times \fp\to G,\quad (k,v)\mapsto k\cdot\exp(v).
	$$
	It follows that any tangent space of $K\backslash G$ (and hence, any tangent space of $M$) is modeled by $\fp$, and the Riemannian metric on it gives an $\Ad(K)$-invariant inner product on $\fp$. Write $S(\fp):=\{v\in\fp: \|v\|=1\}$. For each point $p=[Kg\Gamma]\in M$, the geodesic ray in $M$ through $p$ in the direction $\xi$ can be written as $$
	\gamma(p,\xi)=\{K\exp(tv)g\Gamma: t\geq 0\}\qquad \text{for some }v\in S(\fp).
	$$
	It follows that 
	\begin{equation}\label{E:unit-identify}
			E^+(p,\infty)\cong\{v\in S(\fp): \{K\exp(tv)g\Gamma: t\geq 0\} \text{ is bounded in }M\}.
	\end{equation}

In particular, in view of the similar properties of HAW sets as (W1)--(W2) (Lemma \ref{L:HAW-Ra}), the following corollary is immediate:

\begin{corollary}
    Let $M=K\backslash G/\Gamma$ be as above and let
$(g_n)_{n\in\NN}$ be any sequence in $G$.  Then the set
\begin{equation*}
 \{v\in S(\fp):
       \{K\exp(tv)g_n\Gamma:t\geq0\}\text{ is bounded in } M \text{ for all }n\in\NN\}.
\end{equation*}
is HAW on $S(\fp)$, and hence is thick in $S(\fp)$.
\end{corollary}

\subsection{Organization of the paper and strategy of the proof}\label{SS:strategy}
The whole paper is organized as follows. In Section \ref{S:pre}, we recall some preliminary knowledge on Schmidt’s game and two of its variants: the hyperplane absolute game and Cantor potential game, especially their winning criterion (see Theorem \ref{T:BNY21}). We also recall the quantitative non-divergence result (see Proposition \ref{P:nondivergence}) for real analytic actions with respect to an Ahlfors regular, absolutely decaying measure. These useful tools have been employed in the previous papers \cite{BNY21,BNY22}.

In Section \ref{S:representation}, we recall Mahler's compactness criterion for general $\QQ$-algebraic homogeneous spaces, and construct the crucial rational representation of a $\QQ$-simple algebraic group from its adjoint representation. The key property of this representation is that \textsl{some power of a root operator has rank one}.

In Section \ref{S:global-local-estimates}, we give the global and local estimates around the $r$-vector $$
\bigwedge^r\left(D^{(k,\ell)}_{t,Y}\tau(h(z))\right)\xi.
$$
Compared to the estimates in \cite[Sections 3--4]{Wu26}, the arguments in this paper, albeit following the same philosophies, are more complicated due to the involvement of the key rational representation constructed above. Readers may refer to \cite[Section 1.3]{Wu26} for the original ideas behind these arguments.

In Section \ref{S:para-winning}, we state and prove Theorem \ref{T:para-winning}, a technical result of independent interest that constitutes the main contribution of this paper. Basically, it establishes a winning property on the product of a parameter space and its corresponding expanding horospherical subgroup, from which all the main theorems in Section \ref{S:introduction} will follow.

In Section \ref{S:reductions}, we present complete proofs of all main results in this paper, by performing the standard reductions as in \cite[Section 5]{KM96} and \cite[Sections 3--4]{GW26}. Here a detailed proof of the rank-$1$ semisimple case is included in Appendix \ref{A:Dani}.

\section{Preliminaries}\label{S:pre}

\subsection{Schmidt's game and the hyperplane absolute game}\label{SS:Schmidt} In this subsection, we recall the definitions and properties of Schmidt's game and the hyperplane absolute game.

\begin{definition}
Schmidt's game \cite{Sch66} involves two parameters $\alpha,\beta\in(0,1)$  and is played by two players Alice and Bob on a complete metric space $(X,\dist)$ with a target set $S\subset X$. Bob starts the game by choosing a closed ball $B_0$ of radius $\rho_0$. After Bob chooses a closed ball $B_q$ of radius $\rho_q$, Alice chooses a ball $A_q\subseteq B_q$ of radius ${\rho}'_q=\alpha {\rho}_q$, and then Bob chooses a ball $B_{q+1}\subseteq A_q$ of radius ${\rho}_{q+1}=\beta {\rho}'_{q}$. This implies that the balls are nested:
$$
	B_0 \supset A_0 \supset B_1 \supset\cdots.$$
Alice wins the game if the unique point in $\bigcap_{q=0}^\infty A_q=\bigcap_{q=0}^\infty B_q$ belongs to $S$. The set $S$ is \textsl{$(\alpha,\beta)$-winning} if Alice has a winning strategy, is \textsl{$\alpha$-winning} if it is $(\alpha,\beta)$-winning for any $\beta\in(0,1)$, and is \textsl{winning} if it is $\alpha$-winning for some $\alpha\in (0,1)$.
\end{definition}

Regarding Schmidt's game and its winning sets, Schmidt \cite{Sch66} proved that:	
\begin{itemize}
	\item[(W1)]  winning subsets of Riemannian manifolds are thick;
    \item[(W2)]  a countable intersection of $\alpha$-winning sets is again $\alpha$-winning;
    \item[(W3)]  the image of a winning set under a bi-Lipschitz map is winning.
\end{itemize}
Schmidt's game has been a powerful tool for proving thickness of intersections of certain countable families of sets, see e.g.\ \cite{An13, An16,  BBFKW10, BFK11, BFKRW12}. However, for a fixed $\alpha\in (0,1)$, the class of $\alpha$-winning subsets of a Riemannian manifold depends on the choice of the metric, and is not known to be preserved by diffeomorphisms.


One significant variant is the \textsl{hyperplane absolute game}, which has the advantage that the family of its winning sets is preserved by $C^1$ diffeomorphisms.

\begin{definition}
    Hyperplane absolute game \cite{BFKRW12} involves one parameter $\beta\in (0,\frac{1}{3})$ and is played by two players Alice and Bob on the Euclidean space $\RR^a$ with a target set $S \subset \RR^a$. Bob starts the game by choosing a closed ball $B_0$ of radius ${\rho}_0$.
For an affine hyperplane $L\subset\RR^a$ and $r>0$, we denote the $r$-neighborhood of $L$ by
$$L^{(r)}:=\{z\in\RR^a:\mathrm{dist}(z,L)\leq r\}.$$
After Bob chooses a closed ball $B_q$ of radius ${\rho}_q$, Alice chooses a hyperplane neighborhood $L_q^{(r_q)}$ with $r_q\le\beta {\rho}_q$, and then Bob chooses a closed ball $B_{q+1}\subset B_q\setminus L_q^{(r_q)}$ of radius ${\rho}_{q+1}\ge\beta {\rho}_q$. 

We say that Alice wins by default if at any time Bob has no legal move, or $\rho_q \not\to 0$ as $q\to +\infty$. Otherwise, Alice wins the game if and only if
$$\bigcap_{q=0}^{+\infty} B_q\cap S\ne\emptyset.$$ The set $S$ is \textsl{$\beta$-hyperplane absolute winning} (\textsl{$\beta$-HAW}) if Alice has a winning strategy, and is \textsl{hyperplane absolute winning} (\textsl{HAW}) if it is $\beta$-HAW for any $\beta\in(0,\frac{1}{3})$.
\end{definition}

\begin{lemma}[\cite{BFKRW12}]\label{L:HAW-Ra}
	\begin{itemize}
		\item[(i)] HAW subsets   are winning, and hence thick.
		\item[(ii)] A countable intersection of HAW subsets   is again HAW.
		\item[(iii)]   The image of an HAW set under a $C^1$ diffeomorphism $\RR^a \to \RR^a$ is HAW.
	\end{itemize}
\end{lemma}

The notion of HAW sets has been extended to subsets of $C^1$ manifolds in \cite{KW15}.
	This is done in two steps. First one defines the hyperplane absolute game on an open subset $U\subset \RR^a$. It is defined just as the hyperplane absolute game on
	$\RR^a$, except for requiring that Bob's first move $B_0$ be
	contained in $U$. If Alice has a winning strategy, we say that $S$ is {\sl HAW on $U$}.
	Now let $M$ be an $a$-dimensional $C^1$ manifold, and let $\{(U_\alpha, \varphi_\alpha)\}$ be a $C^1$ atlas, that is, $\{U_\alpha\}$ is an open cover of $M$, and each $\varphi_\alpha$ is a $C^1$ diffeomorphism from $U_\alpha$ onto the open subset $\varphi_\alpha(U_\alpha)$ of $\RR^a$. A subset $S\subset M$ is said to be \textsl{HAW on $M$} if for each $\alpha$, $\varphi_\alpha(S\cap U_\alpha)$ is HAW on $\varphi_\alpha(U_\alpha)$. This definition is independent of the choice of the atlas (see \cite[Corollary 3.4]{KW15}).
	
	\begin{lemma}\label{L:HAW-mnfd}
		\begin{itemize}
			\item[(i)]  HAW subsets of a $C^1$ manifold are thick.
			\item[(ii)]  A countable intersection of HAW subsets of a $C^1$ manifold is again HAW.
			\item[(iii)] Let $f:M\to N$ be a diffeomorphism between $C^1$ manifolds, and let $S\subset M$ be an HAW subset of $M$. Then $f(S)$ is an HAW subset of $N$.
            \item[(iv)] Let $M$ be a $C^1$ manifold, and $M'\subset M$ a finite closed union of $C^1$ submanifolds of strict lower dimensions. Then $S\subseteq M\setminus M'$ is HAW on $M\setminus M'$ if and only if it is HAW on $M$.
            \item[(v)] Let $M, N$ be $C^1$ manifolds, and let $S\subset M$ be an HAW subset of $M$. Then $S\times N$ is an HAW subset of $M\times N$.
		\end{itemize}
	\end{lemma}
	
	\begin{proof}
		(i)--(iii) appeared as \cite[Proposition 3.5]{KW15} and are clear from Lemma \ref{L:HAW-Ra}; 
        (iv) is \cite[Lemma 2.3]{GW26}; (v) is proved in \cite[Proof of Theorem 3.6(a)]{KW15}.
	\end{proof}

\subsection{Cantor potential game}\label{SS:Cantor} In this subsection, we recall another useful variant of Schmidt's game--Cantor potential game, which will play a pivotal role in this paper.

\begin{definition}
    Cantor potential game \cite{BH17,BHNS18} involves two parameters $\alpha\geq 0$, $\beta\in(0,1)$,  and is played by two players Alice and Bob on a complete metric space $(X,\dist)$ with a target set $S\subset X$. Bob starts the game by choosing a closed ball $B_0$ of radius $\rho_0$ in $X$. After Bob chooses a closed ball $B_q$ of radius $\rho_0\beta^q$, Alice chooses, for each $i\geq 0$, a collection $\cA_{q+1,i}$ of at most $\beta^{-\alpha(i+1)}$ balls of radius $\rho_0\beta^{q+1+i}$. Then Bob chooses a ball $B_{q+1}$ of radius $\rho_0\beta^{q+1}$ such that $$
    B_{q+1}\subseteq B_q\setminus \bigcup_{0\leq q'\leq q}\bigcup_{A\in \cA_{q'+1,q-q'}}A.
    $$

    We say that Alice wins by default if at any time Bob has no legal move. Otherwise, Alice wins the game if and only if 
    $$\bigcap_{q=0}^{+\infty} B_q\cap S\ne\emptyset.$$ The set $S$ is \textsl{$(\alpha,\beta)$-Cantor winning} if Alice has a winning strategy, and is \textsl{$\alpha$-Cantor winning} if it is $(\alpha,\beta)$-Cantor winning for any $\beta\in(0,1)$.
\end{definition}

It was shown in \cite{BH17} that on certain nice spaces (e.g. doubling spaces) the properties (W1)--(W3) still hold for Cantor winning sets. In this paper, we shall use another interesting property of Cantor winning sets. To do this, we need to recall the concept of \textsl{Ahlfors regular absolutely decaying measure}.

\begin{definition}\label{D:Ahlfors}
    Let $X$ be a metric space. Given $\alpha>0$, a Borel measure $\nu$ on $X$ is \textsl{$\alpha$-Ahlfors regular}, if there exist $A_0,r_0>0$ such that for any $x\in\supp(\nu)$ and any $r\in (0,r_0)$, 
$$A_0^{-1}r^{\alpha}\leq \nu(B(x,r))\leq A_0r^{\alpha}.$$
We say that $\nu$ is \textsl{Ahlfors regular} if it is $\alpha$-Ahlfors regular for some $\alpha>0$.
\end{definition}

\begin{definition}
    A Borel measure $\nu$ on $\RR^a$ is \textsl{absolutely decaying} if there exist $E_0,\delta,\rho_0>0$ such that for any $x\in\supp(\nu)$, any $\rho\in (0,\rho_0)$, every affine hyperplane $L\subseteq \RR^a$ and any $r>0$, one has $$
    \nu(B(x,\rho)\cap L^{(r)})\leq E_0(r/\rho)^{\delta}\cdot \nu(B(x,\rho)).
    $$
\end{definition}

We list here several external results for our later use:

\begin{lemma}[{\cite[Lemma 22]{BNY21}}]\label{L:BNY-cover}
    Let $\nu$ be an $\alpha$-Ahlfors regular measure on a metric space $X$, let $A_0,r_0>0$ be as in Definition \ref{D:Ahlfors}, and let $S\subseteq X$ be any subset. Suppose that $r\in (0,r_0)$ and $$\nu(S^{(r)}):=\nu\{x\in X:\dist(x,S)\leq r\}<+\infty.$$ Then there exists a cover of $S\cap\supp(\nu)$ by balls with centers in $S\cap\supp(\nu)$ and radius $2r$, whose cardinality is at most $A_0\nu(S^{(r)})\cdot r^{-\alpha}$.
\end{lemma}

\begin{theorem}[{\cite[Proposition 11, Theorem 21]{BNY21}}]\label{T:BNY21}
\,\begin{enumerate}
    \item[(i)] A Borel subset of $\RR^a$ that intersects the support of every nonzero compactly supported Ahlfors regular absolutely decaying measure is HAW.
    \item[(ii)] Let $X$ be the support of a nonzero $\alpha$-Ahlfors regular measure, and let $S\subseteq X$ be $\alpha'$-Cantor winning for some $0<\alpha'<\alpha$. Then $S\neq\varnothing$.
\end{enumerate}
\end{theorem}

\subsection{Exterior powers and quantitative non-divergence}\label{SS:nondivergence}
The theory of quantitative non-divergence originates in Margulis's work \cite{Mar71} on the non-divergence of unipotent trajectories in the space of lattices, and Dani's subsequent recurrence estimates \cite{Dan79,Dan86}. In \cite{KM98}, Kleinbock and Margulis developed the $(C,\alpha)$-good function method and established the quantitative non-divergence estimates with power-law decay. Later, this framework was extended to certain non-Lebesgue measures including friendly measures \cite{KLW04}. These developments provide the quantitative non-divergence arguments employed in \cite{BNY21,BNY22}.

    To state the quantitative non-divergence result in the matrix setting considered here, we first introduce some notations of norms on exterior powers of Euclidean spaces and matrix spaces. Let $a\in\NN$ and let $\la\cdot,\cdot\ra$ denote the standard inner product on $\RR^a$. It induces a norm $\|\cdot\|$ on $\RR^a$ and hence a norm on the exterior powers $\bigwedge^r\RR^a\,(0\leq r\leq a)$.

For a linear map $T:V\to V$, its exterior action is denoted by $\bigwedge^rT$. Its \textsl{infinitesimal exterior action} is denoted by $T^{[r]}$ and is defined by
$$
T^{[r]}(v_1\wedge\cdots\wedge v_r)
:=\sum_{i=1}^r v_1\wedge\cdots\wedge Tv_i\wedge\cdots\wedge v_r.
$$
Thus $\bigwedge^r(\exp T)=\exp(T^{[r]})$. These two actions are kept distinct throughout the proof. 

For a lattice $\Lambda\subseteq V$, put
$$\lambda_1(\Lambda):=\min_{v\in\Lambda\setminus\{0\}}\|v\|.$$
A nonzero exterior vector is \textsl{decomposable} if it is a single wedge of vectors. 

\begin{lemma}[\cite{Cassel}]\label{L:Minkowski}
    There exists a constant $c_1>0$ with the following property. Let $1\leq r<d$ and $\Lambda$ a lattice of $\RR^d$ with covolume $1$. If there exists a rank-$r$ subgroup of $\Lambda$ with covolume $\Delta$ in its real linear span, then $\lambda_1(\Lambda)\leq c_1\Delta^{1/r}$.
\end{lemma}

\begin{proposition}[{\cite[Theorem 5.11]{BNY22}}]\label{P:nondivergence}
    Let $h_0:\RR^a\to\SL_d(\RR)$ be a real analytic map, and let $\nu$ be a nonzero compactly supported Ahlfors regular absolutely decaying measure. There exist $C_0,\alpha_0,r_0>0$ such that, for every ball $B$ centered in $\supp(\nu)$ with radius at most $r_0$, every diagonal $D\in\SL_d(\RR)$, and every $0<\rho\leq1$, at least one of the following holds:
    \begin{enumerate}
        \item[(i)] for every $\epsilon>0$,
        $$
        \nu\{z\in B:\lambda_1(Dh_0(z)\ZZ^d)<\epsilon\}
        \leq C_0(\epsilon/\rho)^{\alpha_0}\nu(3B);
        $$
        \item[(ii)] there exist $1\leq r<d$ and a nonzero 
        decomposable vector $\xi_0\in\bigwedge^r\ZZ^d$ such that
        $$\sup_{z\in B}\left\|\bigwedge^r(Dh_0(z))\xi_0\right\|<\rho.$$
    \end{enumerate}
\end{proposition}

\section{The Arithmetic Representation}\label{S:representation}

\subsection{A general Mahler's compactness criterion}\label{SS:properness}
We first recall a generalization of the classical Mahler's compactness criterion. Readers may refer to \cite{Benoist09, Morris15} for basic definitions and properties of algebraic groups and arithmetic lattices.

\begin{lemma}\label{L:Mahler} Let $\bG$ be a connected $\QQ$-simple algebraic group, let $G=\bG(\RR)^0$, and suppose that $G$ has no compact factors. Let $\Gamma\leq G$ be an arithmetic lattice. Let $\tau:\mathbf G\to\SL(V)$ be a nontrivial rational representation. Then after passing to a subgroup of finite index in $\Gamma$, there exists a full lattice $\cL\subseteq V$ preserved by 
$\Gamma$, such that the map \begin{equation}\label{E:embedding}
\Phi:G/\Gamma\to\SL(V_{\RR})/\SL(\cL),\qquad g\Gamma\mapsto \tau(g)\cL
\end{equation}
is continuous and proper. Here is a natural identification that $$
\SL(V_{\RR})/\SL(\cL)\longleftrightarrow \{\text{lattices in }V_{\RR}\text{ with the same covolume as }\cL\}.
$$
In particular, a subset in $G/\Gamma$ is relatively compact if and only if its corresponding image in the space of lattices is relatively compact.
\end{lemma}
\begin{proof}
Since $\mathbf G$ is $\mathbb Q$-simple and $\tau$ is nontrivial, we see that $\tau$ has finite kernel. Fix a faithful rational realization of $\mathbf G$, and write $\mathbf G_{\mathbb Z}$ for its integral points. By arithmeticity, after passing to a finite-index subgroup of $\Gamma$, we may assume that $\Gamma$ is a finite-index subgroup of $\mathbf G_{\mathbb Z}$. By \cite[Lemma~2.7]{Benoist09}, there is a full lattice $\mathcal L\subset V$ preserved by $\tau(\mathbf G_{\mathbb Z})$, and hence by $\tau(\Gamma)$. So the map \eqref{E:embedding} is well defined and continuous.

Now we show that the map \eqref{E:embedding} is proper. Note that the map \eqref{E:embedding} factors as
$$
G/\Gamma\twoheadrightarrow\tau(G)/(\tau(G)\cap\SL(\cL))\hookrightarrow\mathrm{SL}(V_{\mathbb R})/\mathrm{SL}(\mathcal L),
$$
where the first map is a finite covering by \cite[Lemma~2.17]{Benoist09}, and the last map is a closed embedding by \cite[Proposition 2.9]{Benoist09}. So their composition is proper as desired.
\end{proof}

\subsection{Construction of the rational representation}\label{SS:representation} Now for a connected $\QQ$-simple algebraic group $\bG$, we construct a rational representation $\tau:\bG\to\SL(V)$ to our needs. This approach, suggested by L. Guan in a private conversation, is motivated by the proof of \cite[Proposition 4.2]{Wu26}. Readers may refer to \cite{Hum72, Knapp02} for basic definitions and properties of semisimple Lie groups and their root systems.

\begin{proposition}\label{P:rank-one-power}
Let $\mathbf G$ be a connected $\QQ$-simple algebraic group, let $G=\mathbf G(\RR)^\circ$, and suppose that $G$ has no compact factors. Fix a maximal $\RR$-split torus $A$ whose weight spaces are defined over $\overline{\QQ}\cap\RR$. Let $Y$ be a nonzero element in $\fa=\Lie(A)$. In each real simple factor of $\fg=\Lie(G)$, we choose a system of positive restricted roots such that the corresponding component of $Y$ lies in the closed positive Weyl chamber $\overline{\fa^+}$.

For each real simple factor, let $\theta_i$ be its highest restricted root, and write
$$\lambda=\max_i\theta_i(Y)>0.$$
We choose a highest root $\theta$ in a factor attaining this maximum, and set $m=\dim_{\RR}\fg_\theta$. Then there exist a nontrivial rational irreducible representation $\tau:\bG\to\SL(V)$, a nonzero $E_{\theta}\in\fg_\theta$, an inner product on $V_{\RR}$, and orthonormal vectors $\be,\bff\in V_{\RR}$ with the following property.

For $M:=\mathrm{d}\tau(E_{\theta})$, one has \begin{equation}\label{E:rank-one-power}
M^{2m}=c\cdot\be \bff^{\top}\quad\text{for some }c\neq 0;\qquad
M^{2m+1}=0;
\end{equation}
and 
\begin{equation}
    \mathrm{d}\tau(Y)\be=m\lambda \be,\qquad
\mathrm{d}\tau(Y)\bff=-m\lambda \bff.
\end{equation}
All eigenvalues of $\mathrm{d}\tau(Y)$ lie in $[-m\lambda,m\lambda]$, and $\mathrm{d}\tau(Y)$ is self-adjoint for this inner product. 

Moreover, the representation $\tau$, the element $E_{\theta}$, the inner product on $V_{\RR}$, and an orthonormal weight basis containing $\be,\bff$ can all be chosen with algebraic coordinates, independently of the choices of $Y$.
\end{proposition}

\begin{proof} We divide the proof into several steps.
\medskip

\textsl{Step 1: Construction of $E_{\theta}$ and its $\fsl_2$-triple in a real simple factor.}

We choose a Cartan involution of $\fg$ compatible with $\fa$. This defines an inner product on $\fg$. Fix a $\QQ$-basis of $\fg_{\QQ}$. By our choice of the restricted-root decomposition, we may choose an $E_{\theta}^0\in\mathfrak g_\theta\setminus\{0\}$ whose coordinates in this basis are real algebraic numbers. By \cite[Proposition 6.52]{Knapp02}, there exists a scalar multiple $E_{\theta}$ of $E_{\theta}^0$ in $\fg_{\theta}$, an $F_{\theta}\in \fg_{-\theta}$, and an $H_{\theta}\in \fa$, such that $(E_{\theta},F_{\theta},H_{\theta})$ is an \textsl{$\fsl_2$-triple} in a real simple factor of $\fg$, that is, $$
[E_{\theta},F_{\theta}]=H_{\theta},\quad [H_{\theta},E_{\theta}]=2E_{\theta},\quad [H_{\theta},F_{\theta}]=-2F_{\theta}.
$$
Moreover, these three vectors still have real algebraic coordinates. Note that the eigenvalues of $\ad H_\theta$ on this real simple factor are among $-2,-1,0,1,2$, and its $\pm2$ eigenspaces are precisely $\fg_{\pm\theta}$. On other real simple factors, its eigenvalues are all zero.

Let $M_0:=\ad E_{\theta}$. Then the representation theory of this $\mathfrak{sl}_2$-triple gives
\begin{equation}\label{E:M_0-property}
    M_0^3=0,\quad\text{and}\quad
M_0^2:\fg_{-\theta}\to\fg_\theta
\text{ is a linear isomorphism}.
\end{equation}
It follows that $M_0$ has exactly $m$ Jordan blocks of size three; all its other Jordan blocks have size at most two. 
\medskip 

\textsl{Step 2: Construction of a representation $(\widetilde{\tau},\widetilde{V})$ and the orthonormal vectors $\be,\bff\in \widetilde{V}_{\RR}$.}

Consider the nontrivial rational representation $$
\widetilde{\tau}:=\bigwedge^{m}\Ad:\bG\to \SL(\widetilde{V}),\quad\text{where}\quad \widetilde{V}:= \bigwedge^{m}\Lie(\bG).
$$
The inner product of $\fg$ induces an inner product on $\widetilde{V}_{\RR}=\bigwedge^{m}\fg$. Then the induced action $\widetilde{M}:=\mathrm{d}\widetilde{\tau}(E_{\theta})$ on $\widetilde{V}_{\RR}$ is explicitly given by
\begin{equation}\label{E:widetilde-M}
    v_1\wedge\cdots \wedge v_m\mapsto \sum_{i=1}^m v_1\wedge\cdots\wedge M_0(v_i)\wedge\cdots\wedge v_m.
\end{equation} 
Choose and fix unit vectors $\be$ and $\bff$ in $\bigwedge^m\fg_{\theta}$ and $\bigwedge^m\fg_{-\theta}$ respectively. It is straightforward to see from \eqref{E:M_0-property} and \eqref{E:widetilde-M} that \begin{equation}\label{E:widetilde-M-power}
\widetilde{M}^{2m}\bff=c\cdot \be\quad \text{for some }c\neq 0;\qquad \widetilde{M}^{2m+1}=0.
\end{equation} 

Moreover, we claim that: $\widetilde{M}^{2m}\big|_{\bff^{\perp}}=0$. In fact, consider the induced action $\mathrm{d}\widetilde{\tau}(H_{\theta})$ on $\widetilde{V}_{\RR}$: it is self-adjoint with respect to the inner product, and hence gives an orthogonal direct sum decomposition $$
\widetilde{V}_{\RR}=\bigoplus_{h\in\RR}\widetilde{V}_{h},\qquad \text{where}\quad\widetilde{V}_{h}:=\ker(\mathrm{d}\widetilde{\tau}(H_{\theta})-h\cdot\id).
$$
Its highest weight is $2m$ on $\widetilde{V}_{2m}=\bigwedge^m\fg_{\theta}$, and its lowest weight is $-2m$ on $\widetilde{V}_{-2m}=\bigwedge^m\fg_{-\theta}$. It follows that $$\bff^{\perp}=\bigoplus_{\lambda>-2m}\widetilde{V}_{\lambda}.$$
So our claim again follows from \eqref{E:M_0-property} and \eqref{E:widetilde-M}.

Consequently, we conclude from \eqref{E:widetilde-M-power} and our claim that
$$
\widetilde{M}^{2m}=c\cdot\be \bff^{\top}.
$$

\textsl{Step 3: Construction of an irreducible component $(\tau,V)$.}

By complete reducibility over $\QQ$, we decompose $\widetilde{V}$ as a direct sum of rational irreducible $\mathbf G$-subrepresentations. Since the action $\mathrm{d}\widetilde{\tau}(H_{\theta})$ on $\widetilde{V}_{\RR}$ commutes with all projections onto irreducible components, we see that any eigenspace $\widetilde{V}_{h}$ inherits this direct sum decomposition. In particular, the one-dimensional eigenspace $\widetilde{V}_{-2m}$ must be contained in one of these direct summands, say $V_{\RR}$. 

Moreover, since the action $\widetilde{M}=\mathrm{d}\widetilde{\tau}(E_{\theta})$ also commutes with all projections onto irreducible components, we see that $\widetilde{V}_{2m}=\widetilde{M}^{2m}(\widetilde{V}_{-2m})\subseteq V_{\RR}$. Therefore, the desired irreducible component is $$
\tau:=\widetilde{\tau}\big|_{V}:\bG\to\SL(V).
$$
It is clear that $\be,\bff$ are orthonormal vectors in $V_{\RR}$ satisfying \eqref{E:rank-one-power}.
\medskip 

\textsl{Step 4: Verification of weights of $\mathrm{d}\tau(Y)$.}

Since $Y\in\fa$, the action $\ad(Y)$ is self-adjoint with respect to the inner product on $\fg$. It follows that the induced action $\mathrm{d}\tau(Y)$ is also self-adjoint. For each restricted root $\alpha$ of each real simple factor with its highest restricted root $\theta_i$, one has $$
|\alpha(Y)|\leq \theta_i(Y)\leq \lambda.
$$ 
By definition, the weights of $\mathrm{d}\widetilde{\tau}$ are sums of $m$ adjoint weights, each of which is a restricted root or zero. So all eigenvalues of $\mathrm{d}\widetilde{\tau}(Y)$ lie in $[-m\lambda,m\lambda]$. Moreover, $\mathrm{d}\widetilde{\tau}(Y)\be=m\lambda\be$, $\mathrm{d}\widetilde{\tau}(Y)\bff=-m\lambda\bff$. This completes the proof.
\end{proof}

\section{The Global and Local Estimates}\label{S:global-local-estimates}

This section is devoted to proving the global and local estimates parallel to \cite[Sections 3-4]{Wu26}. 

We first fix some notations. Throughout this section, we retain the arithmetic setting and the representation constructed in Section \ref{S:representation}. Let $\bG,G,A,Y,\theta_i$, and $\theta$ be as in Proposition \ref{P:rank-one-power}, where we renormalize $Y$ so that $$
1=\theta(Y)=\max_i\theta_i(Y).
$$
Let $\tau:\bG\to\SL(V)$ be the rational representation given by Proposition \ref{P:rank-one-power}, and write $d=\dim_{\RR}V_{\RR}$.

We introduce the parameter space as follows. Let $\cO$ be a relatively open subset of an affine subspace of $\fa=\Lie(A)$, contained in $$
\left\{
Y'\in\overline{\fa^+}:1=\theta(Y')\geq \max_{i}\theta_i(Y')
\right\}.
$$
We also assume that all restricted roots and all weights occurring in $\bigwedge^rV_{\RR}\,(1\leq r<d)$ have constant signs on $\cO$. Then the expanding horospherical group $H=H(\{\exp(tY')\}_{t>0})$ is independent of $Y'\in \cO$. Its Lie algebra is $$
\fh=\bigoplus_{\alpha(Y')>0}\fg_{\alpha}\quad\text{ for any }Y'\in\cO.
$$
We work on any fixed nonempty compact subset $\cC$ of $\cO$.

Fix a $\QQ$-basis of $\fg_{\QQ}$. We choose a basis of $\fh$ whose coordinates are real algebraic numbers. With respect to this basis, one has the identification $\RR^{\dim\fh}\cong \fh$. For $z\in \RR^{\dim\fh}$, write $$
h(z):=\exp(z)\in H.
$$
For the element $E_{\theta}\in\fg_{\theta}$ given by Proposition \ref{P:rank-one-power}, we choose an inner product on $\fh$ for which the distinct restricted-root spaces are mutually orthogonal, such that $\|E_{\theta}\|=1$.



Fix $b>1$. For the action of the $\RR$-split torus $A$ on $V_{\RR}$, we choose and fix an orthonormal basis of $V_{\RR}$ consisting of weight vectors, in which the first and last ones are $\be$ and $\bff$. For $k,\ell\in\NN_0$, $t\geq 0$, and $Y\in\cC$, write
\begin{subequations}\label{E:diagonal-maps}
\begin{align}
 A_{k,\ell}&:=b^{(k-\ell)/d}\bigl(I_d+(b^{-k}-1)\be\be^{\top}
           +(b^\ell-1)\bff\bff^{\top}\bigr),\\
B_{t,Y}&:=\tau(\exp(t(\log b)Y)),\\
 D_{t,Y}^{(k,\ell)}&:=A_{k,\ell}\cdot B_{t,Y}.
\end{align}
\end{subequations}
All three maps have determinant one.

Let $\Gamma\leq G$ be an arithmetic lattice, and choose the full lattice $\cL\subseteq V$ from Lemma~\ref{L:Mahler}. We write $\widehat{\cL}=\covol(\cL)^{-1/d}\cdot\cL$ to be a unimodular lattice in $V_{\RR}$. 

\subsection{The global estimate} In this subsection, we prove the global estimate for the supremum of $r$-vector norms, which serves as early steps in excluding the possibility (ii) in Proposition \ref{P:nondivergence}.

\begin{proposition}\label{P:global-estimate}
There exists an integer $L_0\geq 1$ and a constant $\mu>0$ with the following property. For any $g\in G$ and any bounded subset $\cB\subseteq \RR^{\dim\fh}$, there exists some constant $c_0=c_0(g,\cB)>0$ such that 
\begin{equation}\label{E:global}
 \sup_{z\in B(z_0,\rho)}
 \left\|\bigwedge^r\!\left(D_{t,Y}^{(k,\ell)}\tau(h(z))\right)\xi\right\|
 \ge c_0\rho^{L_0} b^{\mu t-(k+\ell)}
\end{equation}
for all $z_0\in\cB$, $0<\rho\le1$, $1\leq r<d$, $k,\ell\in\NN_0$, $t\geq 0$, $Y\in\cC$, and all nonzero
decomposable $\xi\in\bigwedge^r(\tau(g)\widehat{\cL})$.
\end{proposition}

 We divide the proof into several steps as in the following paragraphs.

\subsubsection{Step 1: The projections onto positive and nonnegative weight spaces.}

Let $Y\in\cC$. For each $1\leq r<d$, let $\pi_{r,+}$ and $\pi_{r,\ge0}$ be the projections onto the positive and nonnegative $Y$-weight spaces of $\bigwedge^rV_{\RR}$. Since any weight occurring in $\bigwedge^rV_{\RR}$ has constant sign on $\cC$, these two orthogonal projections are independent of $Y$.

Consider the two linear maps from $\bigwedge^rV_{\RR}$ to finite-dimensional polynomial spaces
\begin{subequations}
    \begin{align}
    \cP_{r,+}:v\mapsto \cP_{r,+}(v),\qquad \text{where}\quad \cP_{r,+}(v)(z)&=\pi_{r,+}\bigl(\bigwedge\nolimits^r\tau(h(z))\bigr)v;\\
\cP_{r,\geq 0}:v\mapsto \cP_{r,\geq 0}(v),\qquad \text{where}\quad \cP_{r,\geq 0}(v)(z)&=\pi_{r,\ge0}\bigl(\bigwedge\nolimits^r\tau(h(z))\bigr)v.
    \end{align}
\end{subequations}
They are polynomial because $H$ is a unipotent subgroup. Their degrees are bounded by a common integer $L_0$ as $r$ varies, and their coefficient matrices have algebraic entries (with respect to a $\QQ$-basis of $\bigwedge^rV$ and a basis of weight spaces whose coordinates are in $\overline{\QQ}\cap \RR$).

Let $G^e$ be the connected product of the real simple factors on which the flow $t\mapsto\exp(tY)$ is nontrivial. We show that:

\begin{lemma}\label{L:check-kernels} One has
\begin{equation}\label{E:positive-kernels}
\ker\cP_{r,+}=\left(\bigwedge\nolimits^rV_{\RR}\right)^{G^e};
\qquad
\ker\cP_{r,\ge0}=\{0\}.
\end{equation}
\end{lemma}
\begin{proof}
Since the map $z\mapsto h(z)$ parametrizes $H$, we see that 
\begin{subequations}
    \begin{align*}
    \ker\cP_{r,+}&=\left\{v\in \bigwedge^rV_{\RR}: \pi_{r,+}\bigl(\bigwedge\nolimits^r\tau(h)\bigr)v=0,\;\forall h\in H\right\};\\
\ker\cP_{r,\geq 0}&=\left\{v\in \bigwedge^rV_{\RR}: \pi_{r,\geq 0}\bigl(\bigwedge\nolimits^r\tau(h)\bigr)v=0,\;\forall h\in H\right\}.
    \end{align*}
\end{subequations}
These kernels are invariant under the actions of $H$ and $\exp(\RR Y)$. Moreover, we decompose $\bigwedge^rV_{\RR}$ into a direct sum of $G$-irreducible components. Since the equivariant projection commutes with the weight projections $\pi_{r,+},\;\pi_{r,\geq 0}$ and the actions $\bigwedge^r\tau(h)\;(h\in H)$, we see that these kernels inherit the direct sum decomposition of $\bigwedge^rV_{\RR}$. It suffices to show that: for each $G$-irreducible component $W$ of $\bigwedge^rV_{\RR}$, one has \begin{equation}\label{E:positive-kernels-irred}
    \ker\cP_{r,+}\cap W=\begin{cases}
        W, & \text{ if }G^e\text{ acts trivially on }W\\
        \{0\}, & \text{ otherwise }
    \end{cases};
\qquad
\ker\cP_{r,\ge0}\cap W=\{0\}.
\end{equation}

First, we suppose that $G^e$ acts trivially on $W$. Since both $H$ and $\exp(\RR Y)$ are contained in $G^e$, we see that for any $v\in W$ and any $h\in H$, $$
\pi_{r,+}\bigl(\bigwedge\nolimits^r\tau(h)\bigr)v=\pi_{r,+}(v)=0,
$$
that is, $v\in \ker\cP_{r,+}$. This verifies that $W\subseteq \ker\cP_{r,+}$.

Next, we suppose that $\ker\cP_{r,+}\cap W\neq \{0\}$. Consider the action $\mathrm{d}(\bigwedge^r\tau)(Y)$ on $W$: it gives an eigenspace decomposition $$
W=\bigoplus_{\lambda'\in\RR}W_{\lambda'},\qquad \text{where}\quad W_{\lambda'}:=\ker\left(\mathrm{d}(\bigwedge^r\tau)(Y)\big|_W-\lambda'\cdot\id\right).
$$
Since the subspace $\ker\cP_{r,+}$ is invariant under the $Y$-action, we see that it inherits the above eigenspace decomposition. Take $$\lambda:=\max\{\lambda':\ker\cP_{r,+}\cap W_{\lambda'}\neq \{0\}\}.$$
We claim that $\lambda=\max\{\lambda':W_{\lambda'}\neq \{0\}\}$. In fact, let us take any nonzero $w\in \ker\cP_{r,+}\cap W_{\lambda}$. Since the subspace $\ker{\cP_{r,+}}$ is invariant under the $\fh$-action, by maximality of $\lambda$, we have for any root $\alpha$ with $\alpha(Y)>0$ and any $z\in\fg_{\alpha}$, $$
\mathrm{d}(\bigwedge^r\tau)(z)(w)\in \ker{\cP_{r,+}}\cap W_{\lambda+\alpha(Y)}=\{0\}.
$$
This implies that the $\fh$-action kills $w$. Write $$
\fh^0=\fg_0\oplus\bigoplus_{\alpha(Y)=0}\fg_{\alpha},\qquad \fh^-=\bigoplus_{\alpha(Y)<0}\fg_{\alpha}.
$$
It follows from the irreducibility of $W$ and Poincaré-Birkhoff-Witt Theorem (see \cite[Theorem 3.8]{Knapp02}) that $W=\cU(\fg)w=\cU(\fh^-)\cU(\fh^0)w$, where $\cU(\cdot)$ is the universal enveloping algebra. In particular, we obtain that $\max\{\lambda':W_{\lambda'}\neq \{0\}\}=\lambda$.

Moreover, since $w\in \ker\cP_{r,+}$, we have $\pi_{r,+}(w)=0$, which implies that $\lambda\leq 0$. Combining this with the fact $\tr\left(\mathrm{d}(\bigwedge^r\tau)(Y)\big|_W\right)=0$, we conclude that $$
\mathrm{d}(\bigwedge^r\tau)(Y)\big|_W=0.
$$
Note that the kernel $\left\{X\in\fg: \mathrm{d}(\bigwedge^r\tau)(X)\big|_W=0\right\}$ is an ideal in $\fg$. This forces that $G^e$ acts trivially on $W$.

Finally, we suppose that $\ker\cP_{r,\geq 0}\cap W\neq \{0\}$. Then the same arguments as above yield $$
\max\{\lambda':W_{\lambda'}\neq \{0\}\}<0.
$$
This is impossible because $\tr\left(\mathrm{d}(\bigwedge^r\tau)(Y)\big|_W\right)=0$. This completes the proof.
\end{proof}

\subsubsection{Step 2: A quantitative lower bound $\cP_{r,+}(\xi)$ where $\xi\in \bigwedge^r(\tau(g)\widehat{\cL})$ is decomposable.}\label{SS:quan-lower-bound} Let $1\leq r<d$. We first show that: \begin{equation}\label{E:decomp-nonkernel}
    \cP_{r,+}(\xi)\neq 0\qquad\text{for any nonzero rational decomposable }r\text{-vector }\xi.
\end{equation}
In fact, suppose that there exists some nonzero rational decomposable $r$-vector $\xi$ such that $\cP_{r,+}(\xi)=0$. It follows from Lemma \ref{L:check-kernels} that $G^e$ fixes $\xi$. Consider the $r$-dimensional $\QQ$-subspace $W(\xi):=\{v\in V: v\wedge\xi=0\}$. Then $G^e$ preserves $W(\xi)_{\RR}$.

Consider the almost direct product of all connected, absolutely simple, normal factors of $\bG_{\overline{\QQ}}$, say $\bG_{\overline{\QQ}}=\bG_1\cdots\bG_s.$ Since $G^e$ is nontrivial, we see that for some $1\leq i\leq s$, $$
\Lie(\bG_i)_{\CC}\subseteq \Lie(G^e)_{\CC}.
$$
Let $\bS:=\bStab_{\bG}(W(\xi))$. It is clear that $\bS$ is a $\QQ$-subgroup of $\bG$, and $G^e\subseteq \bS_{\RR}$. Then we have $\Lie(\bG_i)_{\CC}\subseteq \Lie(\bS)_{\CC}$, and hence $(\bG_i)_{\CC}\subseteq \bS_{\CC}$. In particular, $\bG_i\subseteq \bS_{\overline{\QQ}}$. Since $\bG$ is $\QQ$-simple, its absolutely simple factors form one orbit under $\Gal(\overline{\QQ}/\QQ)$. Thus $\bG_{\overline{\QQ}}\subseteq \bS_{\overline{\QQ}}$. Therefore, we see that $\bG=\bS=\bStab_{\bG}(W(\xi))$, which contradicts the irreducibility of $V$.

Now we prove a quantitative lower bound estimate concerning \eqref{E:decomp-nonkernel}:

\begin{lemma}
    Equip each finite-dimensional polynomial space with its coefficient norm $\|\cdot\|_{\coeff}$. Then there exist two constants $C_0>0$ and $d_0\geq 0$ such that for any $1\leq r<d$,
    \begin{equation}\label{E:algebraic-separation}
\|\cP_{r,+}(\xi_0)\|_{\coeff}\geq C_0\|\xi_0\|^{-d_0}
\quad\text{for any nonzero integral decomposable }r\text{-vector }\xi_0.
\end{equation}
\end{lemma}
\begin{proof}
We choose a $\ZZ$-basis $\{u_1,\cdots,u_N\}$ of $\bigwedge^r\cL$ (which is also a $\QQ$-basis of $\bigwedge^rV$), and choose a basis $\{w_1,\cdots,w_s\}$ of $\mathrm{Im}(\pi_{r,+})\subseteq \bigwedge^rV_{\RR}$ whose coordinates are in $\overline{\QQ}\cap \RR$. Write $$
\cP_{r,+}(u_j)(z)=\sum_{i,\alpha}c_{i,\alpha,j}z^{\alpha}\cdot w_i,\qquad \text{where}\quad c_{i,\alpha,j}\in \overline{\QQ}\cap\RR.
$$
For any nonzero integral decomposable $r$-vector $\xi_0$, write $\xi_0=\sum_{j=1}^Nn_ju_j$ where $n_j\in\ZZ$. It follows from \eqref{E:decomp-nonkernel} that $$
0<\|\cP_{r,+}(\xi_0)\|_{\coeff}=\max_{i,\alpha}\left|\sum_{j=1}^Nn_jc_{i,\alpha,j}\right|.
$$
Let $K_0\subseteq\RR$ be a number field containing all $c_{i,\alpha,j}$. Then by Liouville's inequality (see \cite[Lemma 3.14]{Waldschmidt92}), there exist two constants $C_0>0$ and $d_0=[K_0:\QQ]-1\geq 0$ such that $$
\left|\sum_{j=1}^Nn_jc_{i,\alpha,j}\right|\geq C_0\left(\sum_{j=1}^N|n_j|\right)^{-d_0}\qquad \text{for any }c_{i,\alpha,j}\text{ with }\sum_{j=1}^Nn_jc_{i,\alpha,j}\neq 0.
$$
This proves \eqref{E:algebraic-separation}.\end{proof}

Finally, write $\xi=T(\xi_0)$, where $T:=\covol(\cL)^{-r/d}\cdot \bigwedge^r\tau(g)$, and $\xi_0\in \bigwedge^r\cL$ is nonzero and decomposable. It follows from \eqref{E:algebraic-separation} that $\|\cP_{r,+}(\xi_0)\|_{\coeff}\geq C_0\|\xi_0\|^{-d_0}$. Since $T$ is an invertible linear map on $\bigwedge^rV_{\RR}$ preserving $\ker\cP_{r,+}$, we have 
$$
\|\cP_{r,+}(\xi)\|_{\coeff}\asymp \dist(\xi, \ker(\cP_{r,+}))\geq \|T^{-1}\|^{-1}\cdot \dist(\xi_0, \ker(\cP_{r,+}))\asymp \|T^{-1}\|^{-1}\cdot\|\cP_{r,+}(\xi_0)\|_{\coeff}.
$$
Therefore, we conclude that $\|\cP_{r,+}(\xi)\|_{\coeff}\gg_{g} \|\xi\|^{-d_0}$.

Since $\cP_{r,\geq 0}$ is an injective linear map by Lemma \ref{L:check-kernels}, we have $\|\cP_{r,\geq 0}(\xi)\|_{\coeff}\gg \|\xi\|$.

\subsubsection{Step 3: The supremum norm over a ball $B(z_0,\rho)$ and the action of $D_{t,Y}^{(k,\ell)}$.} Note that for any vector-valued polynomial $P$ with degree at most $L_0$, $$
\sup_{z\in B(z_0,\rho)}\|P(z)\|\geq c_{\cB}\rho^{L_0}\|P\|_{\coeff},\qquad \forall z_0\in \cB,\;\rho\in (0,1].
$$
In particular, by Section \ref{SS:quan-lower-bound}, for any nonzero decomposable $\xi\in \bigwedge^r(\tau(g)\widehat{\cL})$, one has \begin{subequations}\label{E:global-ball-supnorm}
    \begin{align}
        \sup_{z\in B(z_0,\rho)}\|\cP_{r,+}(\xi)(z)\|&\gg_{g,\cB}\rho^{L_0}\|\xi\|^{-d_0};\\ 
        \sup_{z\in B(z_0,\rho)}\|\cP_{r,\geq 0}(\xi)(z)\|&\gg_{\cB}\rho^{L_0}\|\xi\|.
    \end{align}
\end{subequations}

Let $\cW_+$ denote the finite collection of weights in the exterior powers that are positive throughout $\cC$. Set
$$
\mu_0:=\min_{\chi\in \cW_+}\min_{Y\in\cC}\chi(Y)>0.
$$
Consider the action of $\bigwedge^rB_{t,Y}$ on $\mathrm{Im}(\pi_{r,+})$ and $\mathrm{Im}(\pi_{r,\geq 0})$: their eigenvalues are at least $b^{\mu_0 t}$ and $1$, respectively. Then by \eqref{E:global-ball-supnorm}, we see that \begin{equation}
    \sup_{z\in B(z_0,\rho)}\left\|\bigwedge^r(B_{t,Y}\cdot \tau(h(z)))\xi\right\|\gg_{g,\cB}\rho^{L_0}\max\left\{b^{\mu_0 t}\|\xi\|^{-d_0},\;\|\xi\|\right\}.
\end{equation}
Furthermore, note that for any $v\in\bigwedge^rV_{\RR}$, one has $\|\bigwedge^rA_{k,\ell}v\|\geq b^{-(k+\ell)}\|v\|$. We conclude that \begin{equation*}
    \begin{aligned}
        \sup_{z\in B(z_0,\rho)}\left\|\bigwedge^r(D^{(k,\ell)}_{t,Y}\tau( h(z)))\xi\right\|&\gg_{g,\cB} b^{-(k+\ell)}\rho^{L_0}\cdot \max\left\{b^{\mu_0 t}\|\xi\|^{-d_0},\;\|\xi\|\right\}\\
        &\geq \rho^{L_0}b^{\mu_0t/(d_0+1)-(k+\ell)}.
    \end{aligned}
\end{equation*}
This is as desired in \eqref{E:global}. The proof is complete.

\subsection{The local estimate}\label{S:local-estimate}

In this subsection, we prove the local estimate for the supremum of $r$-vector norms, which serves as inductive steps in excluding the possibility (ii) in Proposition \ref{P:nondivergence}. Let $c_1>0$ be a constant given by Lemma \ref{L:Minkowski}. 

\begin{proposition}\label{P:local-estimate}
There exists a constant $C_1>0$ with the following property. 
    Let $B=B(z_0,\rho)$ be a ball in $\RR^{\dim\fh}$, let $k,\ell\in\NN_0$, $t\geq 0$, $Y\in\cC$, and $g\in G$. Suppose that there exists some $1\leq r<d$ and some nonzero decomposable $r$-vector $\xi\in \bigwedge^r(\tau(g)\widehat{\cL})$ such that \begin{equation}\label{E:local-condition}
        \sup_{z\in 2B}\left\|\bigwedge^r\left(D^{(k,\ell)}_{t,Y}\tau(h(z))\right)\xi\right\|\leq 1.
    \end{equation}
    Let $h\geq 0$. If $\rho b^{t-(k+\ell)-2h}\geq C_1$, then for any $z\in B$, \begin{equation}\label{E:local-conclusion}
        \min\left\{
        \lambda_1\left(D^{(k+h,\ell)}_{t,Y}\tau(h(z)g)\widehat{\cL}\right),\;\lambda_1\left(D^{(k,\ell+h)}_{t,Y}\tau(h(z)g)\widehat{\cL}\right)
        \right\}\leq \sqrt{2}c_1b^{-h/d(d-1)}.
    \end{equation} 
\end{proposition}

We divide the proof into several steps as in the following paragraphs.

\subsubsection{Step 1: Alternative for shrinking norms}\label{SS:linear-algebra}

We isolate some basic facts in exterior algebra, and make a useful observation concerning an alternative for shrinking vector norms under two different diagonal actions. This explains the mechanism behind Proposition \ref{P:local-estimate}.

\begin{lemma}\label{L:offdiagonal-power}
Let $V$ be an Euclidean space, $T\in\operatorname{End}(V)$, and let $P$ be an orthogonal projection onto a subspace of $V$. Then \begin{enumerate}
    \item[(1)] for any $j\in\NN$, one has
$$\|(I-P)T^jP\|_{\mathrm{op}}\leq j\|T\|_{\mathrm{op}}^{j-1}
       \|(I-P)TP\|_{\mathrm{op}}.
$$
\item[(2)] if $\zeta$ is a nonzero decomposable $r$-vector representing this subspace, then
$$
\|\zeta\|\,\|(I-P)TP\|_{\mathrm{op}}\leq \|T^{[r]}(\zeta)\|.
$$
\end{enumerate} 
\end{lemma}
\begin{proof}
Set $Q=I-P$. Since $P^2=P,\;Q^2=Q$, we see that for any $j\in\NN$,
$$
QT^jP=(QTP)(PT^{j-1}P)+(QTQ)(QT^{j-1}P).
$$
This gives $$
\|QT^jP\|_{\op}\leq \|QTP\|_{\op}\cdot \|T\|^{j-1}_{\op}+\|T\|_{\op}\cdot \|QT^{j-1}P\|_{\op}.
$$
Then doing induction on $j$ proves (1). 

To prove (2), we take an orthonormal
basis $w_1,\ldots,w_r$ of $W:=\mathrm{Im}(P)$ and write
$\zeta=c\cdot w_1\wedge\cdots\wedge w_r$ where $c\neq 0$. It follows that $$
T^{[r]}(\zeta)=c\cdot\sum_{i=1}^rw_1\wedge\cdots\wedge (PTw_i+QTw_i)\wedge\cdots\wedge w_r\in \bigwedge^{r}W+ \bigwedge^{r-1}W\wedge W^{\perp}.
$$
The component of $T^{[r]}(\zeta)$ in
$\bigwedge^{r-1}W\wedge W^\perp$ has squared norm
$$
|c|^2\sum_{i=1}^r\|QTw_i\|^2\geq \|\zeta\|^2\|QTP\|_{\op}^2.
$$
Therefore, we conclude that $\|T^{[r]}(\zeta)\|\geq \|\zeta\|\|QTP\|_{\op}$.
\end{proof}

Let $V$ be an Euclidean vector space of dimension $d\ge2$, and let $\be,\bff\in V$ be orthonormal. We introduce one-parameter diagonal subgroups $\{D_{\be}^-(t):t\in\RR\}$ and $\{D_{\bff}^+(t):t\in\RR\}$ in $\SL(V)$ by their actions on $V$:
\begin{subequations}
    \begin{align}
        D_{\be}^-(t)\be=e^{-(d-1)t/d}\be,\qquad
&D_{\be}^-(t)|_{\be^\perp}=e^{t/d}\id_{\be^{\perp}},
\label{E:Dfs}\\
D_{\bff}^+(t)\bff=e^{(d-1)t/d}\bff,\qquad
&D_{\bff}^+(t)|_{\bff^\perp}=e^{-t/d}\id_{\bff^{\perp}}.\label{E:Des}
    \end{align}
\end{subequations}

\begin{lemma}\label{L:product-shrink}
Let $s\geq 0$, $1\leq r<d$, let $\zeta$ be a nonzero decomposable $r$-vector, and let $P$
be the orthogonal projection onto the subspace it represents. Suppose that
\begin{equation}\label{E:product-shrink}
    \|\zeta\|\le1,\qquad
\|\zeta\|\,\|(I-P)\be\|\,\|P\bff\|\leq e^{-2s}.
\end{equation}
Then
$$
\min\left\{
e^{rs/d}\|\bigwedge\nolimits^r D_{\bff}^+(s)\zeta\|,
e^{(d-r)s/d}\|\bigwedge\nolimits^r D_{\be}^-(s)\zeta\|
\right\}<\sqrt2.
$$
\end{lemma}
\begin{proof}
By separating the standard exterior coordinates of $\zeta$ according to whether their basis wedge contains $\be$ or $\bff$, we see from \eqref{E:Dfs}\eqref{E:Des} that
\begin{align*}
e^{2rs/d}\|\bigwedge\nolimits^r D_{\bff}^+(s)\zeta\|^2
 &=\|\zeta\|^2\bigl(1+(e^{2s}-1)\|P\bff\|^2\bigr),\\
e^{2(d-r)s/d}\|\bigwedge\nolimits^r D_{\be}^-(s)\zeta\|^2
 &=\|\zeta\|^2\bigl(1+(e^{2s}-1)\|(I-P)\be\|^2\bigr).
\end{align*}
By \eqref{E:product-shrink}, the minimum of their right-hand sides is at most
$$
\|\zeta\|^2+\|\zeta\|^2(e^{2s}-1)\cdot \min\{\|P\bff\|^2,\|(I-P)\be\|^2\}<2.
$$
This is as desired.
\end{proof}

\subsubsection{Step 2: The polynomial shear and the product shrink} We retain our notations in the beginning of this section. Let $M=\mathrm{d}\tau(E_{\theta})\in \mathrm{End}(V_{\RR})$, $m=\dim_{\RR}\fg_{\theta}$, and let $\be,\bff\in V_{\RR}$ be orthonormal vectors given by Proposition \ref{P:rank-one-power}. Then by \eqref{E:rank-one-power}, one has $$
M^{2m}=c\cdot\be \bff^{\top}\quad\text{for some }c\neq 0;\qquad
M^{2m+1}=0;
$$
Let $b>1$. We have chosen an orthonormal basis of $V_{\RR}$ consisting of weight vectors, in which the first and last ones are $\be$ and $\bff$. We first collect some basic properties about $D^{(k,\ell)}_{t,Y}$-conjugacy of $M$.

\begin{lemma}\label{L:property-M'}
    Let $k,\ell\in\NN_0$, $t\geq 0$, $Y\in\cC$, and write $$
M':=D_{t,Y}^{(k,\ell)}M(D_{t,Y}^{(k,\ell)})^{-1}.
$$
Then we have \begin{enumerate}
    \item $(M')^{2m}=c\cdot b^{2mt-(k+\ell)}\be\bff^{\top}$; $(M')^{2m+1}=0$;
    \item $\|M'\|_{\op}\leq b^t\|M\|_{\op}$.
\end{enumerate}
\end{lemma}
\begin{proof}
   (1) It is clear that the second identity follows from $M^{2m+1}=0$. For the first identity, since $\theta(Y)=1$, we observe that $$
\Ad(\exp(t(\log b)Y))E_{\theta}=b^{t\theta(Y)}E_{\theta}=b^tE_{\theta}.
$$
Then we have \begin{equation}\label{E:Bt-conjugacy}
B_{t,Y}MB_{t,Y}^{-1}=\mathrm{d}\tau\left(\Ad(\exp(t(\log b)Y))E_{\theta}\right)
=b^tM.
\end{equation}
It follows from definition that $$(M')^{2m}=A_{k,\ell}(B_{t,Y}MB_{t,Y}^{-1})^{2m}A_{k,\ell}^{-1}=b^{2mt}c\cdot A_{k,\ell}\be\bff^{\top}A_{k,\ell}^{-1}.$$
Since $A_{k,\ell}$ is self-adjoint, we check that $$
A_{k,\ell}\be\bff^{\top}A_{k,\ell}^{-1}=(A_{k,\ell}\be)(A_{k,\ell}^{-1}\bff)^{\top}=(b^{(k-\ell)/d-k}\be)(b^{-(k-\ell)/d-\ell}\bff)^{\top}=b^{-(k+\ell)}\be\bff^{\top}.
$$
These together yield $(M')^{2m}=c\cdot b^{2mt-(k+\ell)}\be\bff^{\top}$.

(2) By (1), we see that $M\be=0$ and $\bff^{\top}M=0$. Moreover, we see from \eqref{E:Bt-conjugacy} that $$
M'=A_{k,\ell}(B_{t,Y}MB_{t,Y}^{-1})A_{k,\ell}^{-1}=b^t\cdot A_{k,\ell}MA_{k,\ell}^{-1}.
$$
Note that \begin{equation}\label{E:Akl-def}
A_{k,\ell}=b^{(k-\ell)/d}(I_d+(b^{-k}-1)\be\be^{\top})(I_d+(b^{-\ell}-1)\bff\bff^{\top})^{-1}.
\end{equation}
Since $b>1$, we have $\|A_{k,\ell}MA_{k,\ell}^{-1}\|_{\op}\leq \|M\|_{\op}$, and hence $\|M'\|_{\op}\leq b^t\|M\|_{\op}$.
\end{proof}

The following result is the crucial polynomial shrink in proving the local estimate. Basically, it says that a short $r$-vector along a polynomial shear leads to a product shrink.

\begin{lemma}\label{L:polynomial-shrink}
There is a constant $C_1\ge1$ with the following property. Let $k,\ell\in\NN_0$, $t\geq 0$, and $Y\in\cC$. Suppose that there exists some $\rho>0$, some $1\leq r<d$ and a nonzero decomposable $r$-vector $\zeta\in\bigwedge^rV_{\RR}$ such that
\begin{equation}\label{E:polynomial-shrink-condition}
    \sup_{|a|\le\rho}\left\|\bigwedge^r\exp(aM')\zeta\right\|\le1.
\end{equation}
Let $P$ denote the orthogonal projection onto the $r$-dimensional subspace that $\zeta$ represents. Then one has
\begin{equation}\label{E:local-product}
\|\zeta\|\leq 1,\qquad
 \|\zeta\|\,\|(I-P)\be\|\,\|P\bff\|
 \le C_1\rho^{-1}b^{(k+\ell)-t}.
\end{equation}
\end{lemma}

\begin{proof}
It is clear that $\|\zeta\|\leq 1$ by taking $a=0$ in \eqref{E:polynomial-shrink-condition}. By Lemma \ref{L:property-M'} (1), we see that \begin{equation}\label{E:perp-op}
    \begin{aligned}
        \|\zeta\|\,\|(I-P)\be\|\,\|P\bff\|&=\|\zeta\|\,\|(I-P)\be\cdot (P\bff)^{\top}\|_{\op}\\
        &=(|c|\cdot b^{2mt-(k+\ell)})^{-1}\|\zeta\|\,\|(I-P)(M')^{2m}P\|_{\op}.
    \end{aligned}
\end{equation}
It suffices to estimate the right-hand side of \eqref{E:perp-op}. By Lemma \ref{L:offdiagonal-power} and Lemma \ref{L:property-M'} (2), we have \begin{equation}\label{E:initial-estimate-M'}
    \begin{aligned}
        \|\zeta\|\|(I-P)(M')^{2m}P\|_{\op}
        &\leq 2m\|M'\|_{\op}^{2m-1}\cdot \|\zeta\|\|(I-P)M'P\|_{\op}\\
        &\leq 2mb^{(2m-1)t}\|M\|_{\op}^{2m-1}\cdot\left\|(M')^{[r]}(\zeta)\right\|.
    \end{aligned}
\end{equation}
So we only need to estimate the norm $\left\|(M')^{[r]}(\zeta)\right\|$.

Consider the vector-valued polynomial function
$$Q:[-\rho,\rho]\to \bigwedge^rV_{\RR},\qquad a\mapsto\bigwedge^r\exp(aM')\zeta.$$
Since $(M')^{2m+1}=0$, the polynomial $Q(a)$ has degree at most $2mr$. Then in view of Markov's inequality (see \cite[Section 5.1]{BE95}), we have \begin{equation*}
\left\|(M')^{[r]}(\zeta)\right\|=\|Q'(0)\|\leq \frac{(\deg{Q})^2}{\rho}\sup_{|a|\leq \rho}\|Q(a)\|\overset{\eqref{E:polynomial-shrink-condition}}{\leq} \frac{(2mr)^2}{\rho}.
\end{equation*}

Therefore, we conclude that the right-hand side of \eqref{E:perp-op} is not exceeding $$
\frac{(2m)^3(d-1)^2\|M\|_{\op}^{2m-1}}{|c|\rho}\cdot b^{(k+\ell)-t}.
$$
The constant $C_1:=\max\{1,(2m)^3(d-1)^2\|M\|_{\op}^{2m-1}|c|^{-1}\}$ is as desired.
\end{proof}

\subsubsection{Step 3: Proof of Proposition \ref{P:local-estimate}}
Now we are ready to prove Proposition \ref{P:local-estimate}. We choose the constant $C_1>0$ as in Lemma \ref{L:polynomial-shrink}.

Since the highest-root vector $E_{\theta}$ lies in the center of $\fh$, we have for any $z\in\fh$ and $a\in\RR$,
$$
 h(aE_{\theta}+z)=\exp(aE_{\theta})\cdot h(z).
$$
When $z\in B$ and $|a|\le\rho$, the point $aE_{\theta}+z$ lies in $2B$. Write $$\zeta=\bigwedge^r(D_{t,Y}^{(k,\ell)}\tau(h(z)))\xi.$$
Then by \eqref{E:local-condition}, we have \begin{equation*}
\begin{aligned}
    \sup_{|a|\le\rho}\left\|\bigwedge^r\exp(aM')\zeta\right\|
    &=\sup_{|a|\le\rho}\left\|\bigwedge^r\left(D^{(k,\ell)}_{t,Y}\exp(aM)\tau(h(z))\right)\xi\right\|\\
    &=\sup_{|a|\le\rho}\left\|\bigwedge^r\left(D^{(k,\ell)}_{t,Y}\tau(h(aE_{\theta}+z))\right)\xi\right\|
    \\
    &\leq \sup_{z'\in 2B}\left\|\bigwedge^r\left(D^{(k,\ell)}_{t,Y}\tau(h(z'))\right)\xi\right\|\leq 1.
\end{aligned}
\end{equation*}
By Lemma~\ref{L:polynomial-shrink}, this implies that $$
\|\zeta\|\leq 1,\qquad \|\zeta\|\,\|(I-P)\be\|\,\|P\bff\|
 \le C_1\rho^{-1}b^{(k+\ell)-t}.
$$
Then the condition $\rho b^{t-(k+\ell)-2h}\geq C_1$ makes the
above right-hand side at most $b^{-2h}$. Choose $s=h\log b$. Then Lemma~\ref{L:product-shrink} gives $$
\min\left\{
e^{rs/d}\|\bigwedge\nolimits^r D_{\bff}^+(s)\zeta\|,\;
e^{(d-r)s/d}\|\bigwedge\nolimits^r D_{\be}^-(s)\zeta\|
\right\}<\sqrt2.
$$
Note that
$$
 D_{\bff}^+(h\log b)D_{t,Y}^{(k,\ell)}
   =D_{t,Y}^{(k,\ell+h)},\qquad 
   D_{\be}^-(h\log b)D_{t,Y}^{(k,\ell)}
   =D_{t,Y}^{(k+h,\ell)}.
$$
Then we have $$
\min\left\{
b^{rh/d}\left\|\bigwedge\nolimits^r \left(D_{t,Y}^{(k,\ell+h)}\tau(h(z))\right)\xi\right\|,\;
b^{(d-r)h/d}\left\|\bigwedge\nolimits^r \left(D_{t,Y}^{(k+h,\ell)}\tau(h(z))\right)\xi\right\|
\right\}<\sqrt2.
$$

Finally, by Lemma \ref{L:Minkowski}, we have for any diagonal matrix $D$, $$
\lambda_1\left(
D\cdot \tau(h(z)g)\widehat{\cL}
\right)\leq c_1\cdot \left\|
\bigwedge^r(D\tau(h(z)))\xi
\right\|^{1/r}.
$$
It follows that $$
\min\left\{
b^{h/d}\cdot
\lambda_1\left(D_{t,Y}^{(k,\ell+h)}\tau(h(z)g)\widehat{\cL}\right),\;
b^{(d-r)h/(rd)}\cdot
\lambda_1\left(D_{t,Y}^{(k+h,\ell)}\tau(h(z)g)\widehat{\cL}\right)
\right\}
<\sqrt{2}c_1.$$  
The conclusion \eqref{E:local-conclusion} follows since $b^{h/(d(d-1))}\leq \min\{b^{h/d},b^{(d-r)h/(rd)}\}$. The proof is complete.

\subsection{The expanding estimate}

We finish the section with an elementary estimate on the $D^{(k,\ell)}_{t,Y}$-conjugacy of differences in the expanding horospherical subgroup $H$.

\begin{lemma}\label{L:distortion}
Let $\cB\subseteq\RR^{\dim\fh}$ be a bounded subset. Then there are
constants $C_{\cB},\epsilon_{\cB}>0$ such that for any $k,\ell\in\NN_0$, $t\geq 0$, $Y\in\mathcal C$, and any $z,z'\in\cB$ with $b^t\|z'-z\|\le\epsilon_{\cB}$, one has
\begin{equation}\label{E:distortion}
 \left\|
 D_{t,Y}^{(k,\ell)}\tau(h(z'))\tau(h(z))^{-1}
       (D_{t,Y}^{(k,\ell)})^{-1}-I
 \right\|_{\mathrm{op}}
 \le C_{\cB}b^t\|z'-z\|.
\end{equation}
\end{lemma}
\begin{proof}
Write $w=\log(h(z')h(z)^{-1})$. Since $H$ is unipotent, we see that $w$ is a well-defined polynomial in $z$ and $z'$. Write $$
T=\mathrm{d}\tau(\Ad(\exp(t(\log{b})Y))w).
$$
Note that the expression inside the norm in \eqref{E:distortion} is
$$A_{k,\ell}B_{t,Y}\tau(\exp(w))B_{t,Y}^{-1}A_{k,\ell}^{-1}-I
=A_{k,\ell}(\exp(T)-I)A_{k,\ell}^{-1}.$$
In view of \eqref{E:Akl-def}, we see that $\|A_{k,\ell}(\exp(T)-I)A_{k,\ell}^{-1}\|_{\op}\leq \|\exp(T)-I\|_{\op}$.

It is clear that there exists some $C_{\cB}>0$ depending on $\tau:\bG\to\SL(V)$ and the bounded sets $\cC,\cB$ such that for any $t\geq 0$, $Y\in\mathcal C$, and $z,z'\in\cB$, one has $$
\|T\|_{\op}\ll b^t\|w\|\leq C_{\cB}b^t\|z-z'\|.
$$
When we choose $\epsilon_{\cB}\in (0,C_{\cB}^{-1})$, by our condition, the above right-hand side is at most $1$. Then the Taylor's expansion yields
$$
 \|\exp T-I\|_{\op}\le 2\|T\|_{\op}\le 2C_{\cB}b^t\|z-z'\|.
$$
 Combining these estimates gives \eqref{E:distortion}.
\end{proof}

\section{Winning Result on a Parameter Space}\label{S:para-winning}

In this section, we establish a winning result (Theorem \ref{T:para-winning}) on the product of a parameter space and its corresponding expanding horospherical subgroup.

Let $\bG,G,A,\theta_i$, and $\theta$ be as in Proposition \ref{P:rank-one-power}, and let $\cO$ be the parameter space as in the beginning of Section \ref{S:global-local-estimates}. Then the expanding horospherical group $H=H(\{\exp(tY)\}_{t>0})$ is independent of $Y\in \cO$. Let $\Gamma\leq G$ be an arithmetic lattice. Our main result of this section is:

\begin{theorem}\label{T:para-winning}
    For any $x\in G/\Gamma$, the set \begin{equation}\label{E:para-winning}
        \left\{
        (Y,h)\in \cO\times H:\{\exp(tY)hx:t\geq 0\} \text{ is bounded in }X
        \right\}
    \end{equation}
    is HAW on $\cO\times H$.
\end{theorem}

\subsection{Choice of parameters}\label{SS:parameter}
We first choose and fix some parameters. We choose affine Euclidean coordinates $\omega\mapsto Y(\omega)$ on the affine span of $\cO$, and work on a pre-compact open set $\Omega\subset\RR^{\dim\cO}$ such that $Y(\overline\Omega)\subset\cO$. 

Write $\fh=\Lie(H)$. We fix a nonzero compactly supported, $\alpha$-Ahlfors regular, absolutely decaying measure $\nu$ on the product space $\RR^{\dim\cO}\times\fh$, whose support is contained in $\Omega\times\fh$. Let $A_0>0$ be as in Definition \ref{D:Ahlfors}. Write $\cK:=\supp(\nu)$, and let $\cB\subseteq \RR^{\dim\fh}\cong\fh$ be a $1$-neighborhood of the image of the projection of $\cK\subseteq\Omega\times\fh$ onto $\fh$. Let $C_{\cB},\epsilon_{\cB}>0$ be given by Lemma \ref{L:distortion}.

Let $L_0\geq 1$ and $\mu>0$ be given by Proposition \ref{P:global-estimate}. By shrinking $\mu$ if necessary, we may assume that $0<\mu<L_0$. Then we choose an integer $u>\max\{1,\mu^{-1}\}$, and choose
\begin{equation}\label{E:kappa}
 0<\kappa<\min\left\{
 \frac{\mu u-1}{2u(L_0-\mu)},\frac{u-1}{4}\right\}.
\end{equation}
Let $C_0,\alpha_0>0$ be given by Proposition \ref{P:nondivergence}. Then we choose
\begin{equation}\label{E:eta}
 0<\eta<\min\left\{
 \frac{\kappa}{4d(d-1)(1+\kappa)},
 \frac{u\alpha}{\alpha_0}\right\},
\end{equation}
and 
\begin{equation}\label{E:alphap}
 \alpha'=\alpha-\frac{\eta\alpha_0}{4u}.
\end{equation}

Now fix any $\beta\in(0,1)$ and put $b=\beta^{-1}$. Let $r_0>0$ be smaller than those in Definition \ref{D:Ahlfors} and Proposition \ref{P:nondivergence}. Let $\rho_0>0$ be a sufficiently small initial radius so that 
\begin{equation}\label{E:rho_0}
    6\rho_0\leq \min\{1,r_0\},\quad \rho_0\leq 2\min\{C_{\cB}^{-1},\epsilon_{\cB}\}.
\end{equation}
For $q\in\NN_0$, set $\rho_q:=\rho_0b^{-q}$. Let $(B_q)_{q\in\NN_0}$ be a sequence of ambient balls with centers $(\omega_q,z_q)$ in $\cK$ and radius $\rho_q$. We first observe the following fact: 

\begin{lemma}\label{L:rho_0_0}
    The initial radius $\rho_0$ may be chosen sufficiently small so that the following property holds for all $q,k,\ell\in\NN_0$ with $k+\ell\ge1$. Write $t=q+u(k+\ell)$. For any $\omega\in\Omega$ with $\|\omega-\omega_q\|\leq \rho_q$, and any $1\le r< d$, one has
\begin{equation}\label{E:rho_0_0}
 \left\|\bigwedge\nolimits^r\!\left(
 D_{t,Y(\omega)}^{(k,\ell)}(D_{t,Y(\omega_q)}^{(k,\ell)})^{-1}
 \right)^{\!\pm1}\right\|\le b^{\eta (k+\ell)/2}.
\end{equation}
\end{lemma}
\begin{proof}
Note that $$
 D_{t,Y(\omega)}^{(k,\ell)}(D_{t,Y(\omega_q)}^{(k,\ell)})^{-1}
 =\tau(\exp(t(\log b)(Y(\omega)-Y(\omega_q)))),
$$
where $\|Y(\omega)-Y(\omega_q)\|\le L_Y\|\omega-\omega_q\|\leq L_Y\rho_q$ for some $L_Y>0$. Then we see that the operator norms of $D_{t,Y(\omega)}^{(k,\ell)}(D_{t,Y(\omega_q)}^{(k,\ell)})^{-1}$ and its inverse are at most $b^{Ct\rho_q}$, where $C>0$ is a constant only depending on $\tau:\bG\to\SL(V)$ and $Y:\overline{\Omega}\to \cO$. It follows that the left-hand side of \eqref{E:rho_0_0} is at most $b^{(d-1)Ct\rho_q}$.
Since
\[
 t\rho_q=(q+u(k+\ell))\rho_0b^{-q}
 \le\rho_0\left(\sup_{j\ge0}jb^{-j}+u\right)(k+\ell),
\]
we only need to choose $\rho_0>0$ sufficiently small so that $(d-1)C\rho_0(\sup_{j\ge0}jb^{-j}+u)\le\eta/2$. This completes the proof.
\end{proof}

Fix $g\in G$. Then we fix $c_0=c_0(g,\cB)$ as in Proposition \ref{P:global-estimate}, $C_1>0$ as in Proposition \ref{P:local-estimate}, and $c_1>0$ as in Lemma \ref{L:Minkowski}. 

Let $N_0\geq 2/\kappa$ be a sufficiently large integer so that for all $N\geq N_0$, \begin{subequations}
 \begin{align}
 c_0(2\rho_0)^{L_0}b^{(\mu u-1)N/2}&\geq 1;\label{E:N01}\\
\rho_0b^{2\kappa N}&\geq C_1;\label{E:N02}\\
b^{\eta N/2}&\geq \sqrt{2}c_1. \label{E:N03}
 \end{align}
 \end{subequations} 

 \subsection{Dangerous sets and their measure control}\label{SS:dangerous} In this subsection, we define the dangerous sets to be avoided for our target sets and give a local control of their measure. This is based on the quantitative non-divergence result (Proposition \ref{P:nondivergence}), as well as the global and local estimates (Propositions \ref{P:global-estimate} and \ref{P:local-estimate}).

Fix $g\in G$. For $q,k,\ell\in\NN_0$ with $k+\ell\ge N_0$, we define the \textsl{raw $(q;k,\ell)$-dangerous set} as
\begin{equation}\label{par:raw}
 \cD_q^{(k,\ell)}
 =\left\{(\omega,z)\in\Omega\times\fh:
\lambda_1\left(D_{q+u(k+\ell),Y(\omega)}^{(k,\ell)}\tau(h(z)g)\widehat{\cL}\right)
                     <b^{-\eta (k+\ell)}\right\}.
\end{equation}
Given Bob's $q$-th game ball $\cK\cap B_q$, we define the \textsl{current $(q;k,\ell)$-dangerous set} as
\begin{equation}\label{par:current}
 \widetilde{\cD}_q^{(k,\ell)}
 =\bigl(\mathcal K\cap B_q\cap\cD_q^{(k,\ell)}\bigr)
 \setminus
 \bigcup_{0\le q'<q}\ \bigcup_{k'+\ell'\ge N_0}
                  \cD_{q'}^{(k',\ell')}.
\end{equation}
The key ingredient in proving Theorem \ref{T:para-winning} is the following historical alternative:

\begin{proposition}\label{P:history}
For $q,k,\ell\in\NN_0$ with $k+\ell\geq N_0$, write $t=q+u(k+\ell)$. Suppose that there exists $1\le r<d$ and some nonzero decomposable $r$-vector $\xi\in\bigwedge^r(\tau(g)\widehat{\cL})$ such that
\begin{equation}\label{par:frozen-obstruction}
 \sup_{(\omega,z)\in2B_q}
 \left\|
 \bigwedge\nolimits^r\left(D_{t,Y(\omega_q)}^{(k,\ell)}\tau(h(z))\right)\xi
 \right\|<1.
\end{equation}
Write $h:=\lfloor\kappa(k+\ell)\rfloor>0$. Then $q>uh$, and
\begin{equation}\label{par:history-inclusion}
 \cK\cap B_q\subseteq
 \cD_{q-uh}^{(k+h,\ell)}\cup\cD_{q-uh}^{(k,\ell+h)}.
\end{equation}
\end{proposition}
\begin{proof}
Suppose to the contrary that $q\leq uh$. Since $z_q\in\cB$ and $2\rho_q\leq 2\rho_0\leq 1$ by \eqref{E:rho_0}, Proposition \ref{P:global-estimate} gives \begin{equation}\label{E:history-cal}
    \begin{aligned}
        \sup_{(\omega,z)\in2B_q}
 \left\|
\bigwedge\nolimits^r\left(D_{t,Y(\omega_q)}^{(k,\ell)}\tau(h(z))\right)\xi
 \right\|
 &\geq \sup_{z\in B(z_q,2\rho_q)}
 \left\|
\bigwedge\nolimits^r\left(D_{t,Y(\omega_q)}^{(k,\ell)}\tau(h(z))\right)\xi
 \right\|\\
 &\geq c_0(2\rho_q)^{L_0}b^{\mu t-(k+\ell)}\\
 &=c_0(2\rho_0)^{L_0}b^{(\mu-L_0)q+(\mu u-1)(k+\ell)}.
    \end{aligned}
\end{equation} 
Since $\mu<L_0$ and $q\leq u\kappa(k+\ell)$, we see that $$
(\mu-L_0)q+(\mu u-1)(k+\ell)\overset{}{\geq }((\mu-L_0)u\kappa+(\mu u-1))(k+\ell)\overset{\eqref{E:kappa}}{\geq} (\mu u-1)(k+\ell)/2,
$$
and hence the right-hand side of \eqref{E:history-cal} is at least $c_0(2\rho_0)^{L_0}b^{(\mu u-1)(k+\ell)/2}\geq 1$ by \eqref{E:N01}. This contradicts \eqref{par:frozen-obstruction}. Therefore, we have $q>uh$.

Next we verify \eqref{par:history-inclusion}. Let $(\omega,z)\in B_q$. Since \eqref{par:frozen-obstruction} holds and $$
\rho_qb^{t-(k+\ell)-2h}\geq\rho_0b^{(u-1-2\kappa)(k+\ell)}\overset{\eqref{E:kappa}}{>}\rho_0b^{2\kappa(k+\ell)}\overset{\eqref{E:N02}}{\geq }C_1,
$$
Proposition \ref{P:local-estimate} gives \begin{equation*}
    \min\left\{       \lambda_1\left(D^{(k+h,\ell)}_{t,Y(\omega_q)}\tau(h(z)g)\widehat{\cL}\right),\;\lambda_1\left(D^{(k,\ell+h)}_{t,Y(\omega_q)}\tau(h(z)g)\widehat{\cL}\right)
        \right\}\leq \sqrt{2}c_1b^{-h/d(d-1)}.
\end{equation*}
Moreover, by Lemma \ref{L:rho_0_0}, we see that 
\begin{equation}\label{E:min-cal}
    \min\left\{       \lambda_1\left(D^{(k+h,\ell)}_{t,Y(\omega)}\tau(h(z)g)\widehat{\cL}\right),\;\lambda_1\left(D^{(k,\ell+h)}_{t,Y(\omega)}\tau(h(z)g)\widehat{\cL}\right)
        \right\}\leq \sqrt{2}c_1b^{-h/d(d-1)+\eta(k+\ell)/2}.
\end{equation}
Since $\frac{1}{2}\kappa(k+\ell)\leq h\leq \kappa(k+\ell)$, we check that $$
\frac{h}{d(d-1)}\geq \frac{\kappa(k+\ell)}{2d(d-1)}\overset{\eqref{E:eta}}{>}2\eta(1+\kappa)(k+\ell)\geq 2\eta(k+\ell+h),
$$
and hence the right-hand side of \eqref{E:min-cal} is at most $$
\sqrt{2}c_1b^{-2\eta(k+\ell+h)+\eta(k+\ell)/2}\overset{\eqref{E:N03}}{\leq} b^{-\eta(k+\ell+h)}.
$$
This means that $(\omega,z)\in\cD_{q-uh}^{(k+h,\ell)}\cup\cD_{q-uh}^{(k,\ell+h)}.$ The proof is complete.
\end{proof}

Then we easily deduce a quantitative measure control of dangerous sets from a combination of Propositions \ref{P:nondivergence} and \ref{P:history}.

\begin{corollary}\label{C:measure-control}
    For $q,k,\ell\in\NN_0$ with $k+\ell\geq N_0$, write $t=q+u(k+\ell)$. If $\widetilde{\cD}_q^{(k,\ell)}\neq \varnothing$, then for any $\epsilon>0$, one has
\begin{equation}\label{E:measure-bound}
 \nu\left\{(\omega,z)\in2B_q:
 \lambda_1\left(D_{t,Y(\omega_q)}^{(k,\ell)}\tau(h(z)g)\widehat{\cL}\right)<\epsilon\right\}
 \le C_0\epsilon^{\alpha_0}\nu(6B_q).
\end{equation}
\end{corollary}
\begin{proof}
    Write $\widehat{\cL}=\tau(g_0)V_{\ZZ}$ for some $g_0\in G$. Since the map $(\omega,z)\mapsto \tau(h(z)gg_0)$ is real analytic, we apply Proposition \ref{P:nondivergence} to $B=2B_q$ with $2\rho_q\leq r_0$ by \eqref{E:rho_0}, $D=D_{t,Y(\omega_q)}^{(k,\ell)}$, and $\rho=1$. Suppose that the possibility (ii) in Proposition \ref{P:nondivergence} holds, that is, \eqref{par:frozen-obstruction} holds for some $1\le r<d$ and some nonzero decomposable $r$-vector $\xi=\bigwedge^r\tau(gg_0)(\xi_0)$. Then by Proposition \ref{P:history}, one has $$
    \cK\cap B_q\subseteq
 \bigcup_{0<q'<q}\cD_{q'}^{(k+h,\ell)}\cup\cD_{q'}^{(k,\ell+h)},
    $$
    which contradicts $\widetilde{\cD}_q^{(k,\ell)}\neq \varnothing$. Therefore, the possibility (i) in Proposition \ref{P:nondivergence} holds, that is, \eqref{E:measure-bound} holds. The proof is complete.
\end{proof}

\subsection{Cantor potential and Alice's strategy} In this subsection, we verify that the measure control of dangerous sets in Section \ref{SS:dangerous} provides a legal collection of balls for Alice to delete. The number of covering balls is counted as follows:

\begin{proposition}\label{P:covering}
There exists a constant $C_2>0$ with the following property. For any $q,k,\ell\in\NN_0$ with $k+\ell\geq N_0$, the set $\widetilde{\cD}_q^{(k,\ell)}$ admits a cover $\cA^{(k,\ell)}_q$ consisting of balls with centers in $\mathcal K$ and radius $\rho_{q+u(k+\ell)}$, whose cardinality is at most $C_2b^{(u\alpha-\eta\alpha_0/2)(k+\ell)}$.
\end{proposition}
\begin{proof}
If $\widetilde{\cD}_q^{(k,\ell)}=\varnothing$, the conclusion is clear. Now we assume that $\widetilde{\cD}_q^{(k,\ell)}\neq \varnothing$, and write $t=q+u(k+\ell)$. We claim that: the closed $\rho_t/2$-neighborhood of $\widetilde{\cD}_q^{(k,\ell)}$ is contained in the set $$
 \left\{(\omega',z')\in2B_q:
 \lambda_1\left(D_{t,Y(\omega_q)}^{(k,\ell)}\tau(h(z')g)\widehat{\cL}\right)
 <2b^{-\eta (k+\ell)/2}\right\}.
$$
In fact, for any $(\omega,z)\in \widetilde{\cD}_q^{(k,\ell)}$, one has $\|\omega-\omega_q\|\leq \rho_q$, and $$
\lambda_1\left(D_{t,Y(\omega)}^{(k,\ell)}\tau(h(z)g)\widehat{\cL}\right)
<b^{-\eta (k+\ell)}.
$$
Then by Lemma \ref{L:rho_0_0}, we have $$
\lambda_1\left(D_{t,Y(\omega_q)}^{(k,\ell)}\tau(h(z)g)\widehat{\cL}\right)
<b^{-\eta (k+\ell)/2}.
$$
Let $(\omega',z')$ be in the $\rho_t/2$-neighborhood of $(\omega,z)$. Then $b^t\|z-z'\|\leq \rho_0/2\leq \epsilon_{\cB}$ by \eqref{E:rho_0}. It follows from Lemma \ref{L:distortion} that \begin{equation*}
    \begin{aligned}     \lambda_1\left(D_{t,Y(\omega_q)}^{(k,\ell)}\tau(h(z')g)\widehat{\cL}\right)
&\leq \left(1+C_{\cB}b^t\|z-z'\|\right)\cdot \lambda_1\left(D_{t,Y(\omega_q)}^{(k,\ell)}\tau(h(z)g)\widehat{\cL}\right)\\
&<(1+C_{\cB}\rho_0/2) b^{-\eta (k+\ell)/2}.
    \end{aligned}
\end{equation*}
Since \eqref{E:rho_0} gives $1+C_{\cB}\rho_0/2\leq 2$, this verifies our claim.

In view of our claim, Corollary \ref{C:measure-control} and Definition \ref{D:Ahlfors} together imply that the $\nu$-measure of the closed $\rho_t/2$-neighborhood of $\widetilde{\cD}_q^{(k,\ell)}$ is at most $$
C_0(2b^{-\eta (k+\ell)/2})^{\alpha_0}\nu(6B_q)\leq 
C_0(2b^{-\eta (k+\ell)/2})^{\alpha_0}\cdot A_0(6\rho_0b^{-q})^{\alpha}.
$$
Then Lemma \ref{L:BNY-cover} gives a cover of $\widetilde{\cD}_q^{(k,\ell)}$ by balls with centers in $\cK$ and radius $\rho_t$, whose cardinality is at most $$
A_0\cdot \nu\left(\left(\widetilde{\cD}_q^{(k,\ell)}\right)^{(\rho_t/2)}\right)\cdot (\rho_t/2)^{-\alpha}
\leq C_2b^{(u\alpha-\eta\alpha_0/2)(k+\ell)},
$$
where $C_2:=A_0^2C_02^{\alpha_0}12^{\alpha}$. This completes the proof.
\end{proof}

Now we enlarge $N_0$ in Section \ref{SS:parameter}  so that for any $N\ge N_0$,
\begin{equation}\label{E:N04}
 (N+1)C_2\le b^{\eta\alpha_0N/4}.
\end{equation}
We are ready to define Alice’s strategy of deleting collections of balls. For $q\in\NN_0$, let Alice choose her collections $\cA_{q+1,i}\,(i\geq 0)$ as
\begin{equation}\label{par:schedule}
 \cA_{q+1,i}=
 \begin{cases}
 \bigcup_{(k,\ell):k+\ell=N}\cA_q^{(k,\ell)},
     & i=uN-1\text{ for some }N\ge N_0,\\
 \varnothing,&\text{otherwise}.
 \end{cases}
\end{equation}
The announced balls are balls in $\mathcal K$, obtained by intersecting
the ambient balls with $\mathcal K$. They have the required radius
$\rho_{q+1+i}$. Their cardinality is at most
$$
 (N+1)C_2b^{(u\alpha-\eta\alpha_0/2)N}
 \overset{\eqref{E:N04}}{\leq} b^{(u\alpha-\eta\alpha_0/4)N}
 \overset{\eqref{E:alphap}}{=}\beta^{-\alpha'(i+1)}.
$$
Thus Alice's strategy is legal.

\subsection{Proof of Theorem \ref{T:para-winning}} Finally, we prove that Alice’s strategy above is a winning strategy for the target set \eqref{E:para-winning} when restricted to any local chart $Y(\Omega)\times H$, and hence complete the proof of Theorem \ref{T:para-winning} (see the notion of HAW on manifolds in Section \ref{SS:Schmidt}).

Let $\nu$ be a nonzero, compactly supported, $\alpha$-Ahlfors regular, absolutely decaying measure on $\Omega\times\fh$. Write $\cK=\supp(\nu)$. We choose $\alpha'\in (0,\alpha)$ as in \eqref{E:alphap} and let the $\alpha'$-Cantor potential game on $\cK$ go as in previous subsections. Write $x=g\Gamma$ for some $g\in G$.
    
    We first show that: if Alice doesn't win by default, then the outcome of the game avoids all raw dangerous sets $\cD_{q}^{(k,\ell)}$ for $q,k,\ell\in\NN_0$ with $k+\ell\geq N_0$. In fact, let $(\omega_*,z_*)\in \cK\cap\bigcap_{q=0}^{+\infty}B_q$. Suppose that $(\omega_*,z_*)\in \cD^{(k,\ell)}_q$ for some $q,k,\ell$. Then we may choose $q$ minimal so that $$
    (\omega_*,z_*)\in \cD_q^{(k,\ell)}\setminus \bigcup_{q'<q}\bigcup_{k'+\ell'\geq N_0}\cD_{q'}^{(k',\ell')}.
    $$ 
    This implies that $(\omega_*,z_*)\in \widetilde{\cD}^{(k,\ell)}_q$. So $(\omega_*,z_*)$ is contained in some ball in the collection $\cA^{(k,\ell)}_q$, which doesn't meet $\cK\cap B_{q+u(k+\ell)}$. This is a contradiction.

    We next show that: if some point $(\omega_*,z_*)$ avoids all raw dangerous sets $\cD_{q}^{(k,\ell)}$ for $q,k,\ell\in\NN_0$ with $k+\ell\geq N_0$, then $(\omega_*,z_*)$ lies in the following set \begin{equation}\label{E:target-Lie}
        \left\{(\omega,z)\in \Omega\times \fh: 
    \{\exp(tY(\omega))h(z)x:t\geq 0\}\text{ is bounded in }G/\Gamma
    \right\}.
    \end{equation} In fact, for all $q,k,\ell$, one has $$   \lambda_1\left(D^{(k,\ell)}_{q+u(k+\ell),Y(\omega_*)}\tau(h(z_*)g)\widehat{\cL}\right)\geq b^{-\eta(k+\ell)}.
    $$
    Taking $(k,\ell)=(N_0,0)$ with $$A_{N_0,0}\cdot B_{q+uN_0,Y(\omega_*)}=D^{(N_0,0)}_{q+uN_0,Y(\omega_*)}$$
    gives for all $q\in\NN_0$, \begin{equation*}
        \begin{aligned}
       \lambda_1\left(B_{q+uN_0,Y(\omega_*)}\tau(h(z_*)g)\widehat{\cL}\right)
    &\geq \|A_{N_0,0}\|^{-1}\cdot\lambda_1\left(D^{(N_0,0)}_{q+uN_0,Y(\omega_*)}\tau(h(z_*)g)\widehat{\cL}\right)\\
    &\geq b^{-(\eta+1/d)N_0}.
    \end{aligned}
    \end{equation*}
    Moreover, for $0\leq q<uN_0$,  one has \begin{equation*}
        \begin{aligned}
            \lambda_1\left(B_{q,Y(\omega_*)}\tau(h(z_*)g)\widehat{\cL}\right)&\geq \|B_{uN_0,Y(\omega_*)}\|^{-1}\cdot \lambda_1\left(B_{q+uN_0,Y(\omega_*)}\tau(h(z_*)g)\widehat{\cL}\right)\\
            &\gg_{\tau} b^{-(r_{*}u+\eta+1/d)N_0},
        \end{aligned}
    \end{equation*}
    where $r_*>0$ is a constant depending on $Y(\omega_*)\in Y(\overline{\Omega})$. So we conclude that $$
    \inf_{q\in\NN_0}\lambda_1\left(B_{q,Y(\omega_*)}\tau(h(z_*)g)\widehat{\cL}\right)>0.
    $$
    Then it follows from Lemma \ref{L:Mahler} that $\{\exp(tY(\omega_*))h(z_*)x:t\geq 0\}$ is bounded in $G/\Gamma$.

    Therefore, we have shown that the intersection of $\cK$ with the set \eqref{E:target-Lie} is $\alpha'$-Cantor winning. Then Theorem \ref{T:BNY21} (ii) shows that this intersection is nonempty. By Theorem \ref{T:BNY21} (i), the set \eqref{E:target-Lie} is HAW on $\Omega\times\fh$. Finally, it follows from Lemma \ref{L:HAW-mnfd} (iii) that the target set \eqref{E:para-winning} is HAW on $Y(\Omega)\times H$. This completes the proof.

\section{Proofs of Main Theorems}\label{S:reductions}

In this section, we deduce all our main results from Theorem \ref{T:para-winning}, using the standard reductions as in \cite[Section 5]{KM96} and \cite[Sections 3--4]{GW26}.

\subsection{Proofs in the arithmetic setting} Let $\bG$ be a connected $\QQ$-simple algebraic group such that $G=\bG(\RR)^\circ$ has no compact factors, and let $\Gamma<G$ be an arithmetic lattice. Fix a maximal $\RR$-split torus $A$ whose weight spaces are defined over $\overline{\QQ}\cap\RR$, and choose an open positive Weyl chamber $\fa^+$.

We present two statements here which are direct consequences of Theorem \ref{T:para-winning}.

\begin{theorem}\label{T:alg-H}
    For $G,\Gamma$ as above, let $F^+$ be a one-parameter $\Ad$-$\RR$-diagonalizable subsemigroup of $G$, and let $H=H(F^+)$. Then for any $x\in X=G/\Gamma$, the set $$
    E_{H,x}(F^+)=\{h\in H:F^+hx\text{ is bounded in }X\}
    $$
    is HAW on $H$.
\end{theorem}
\begin{proof} If $F^+$ is trivial, then the conclusion is clear. If $F^+$ is nontrivial, since it is $\Ad$-$\RR$-diagonalizable, we may choose some $g_0\in G$ such that $$
g_0F^+g_0^{-1}=\{\exp(tY):t\geq 0\}\quad \text{ for some nonzero }Y\in \overline{\fa^+}.
$$
Note that for any $x\in G/\Gamma$, we have $E_{H,x}(F^+)=g_0^{-1}E_{g_0Hg_0^{-1},g_0x}(g_0F^+g_0^{-1})g_0$. So in view of Lemma \ref{L:HAW-mnfd} (iii), showing that $E_{H,x}(F^+)$ is HAW on $H$ is equivalent to showing that \begin{equation}
    E_{g_0Hg_0^{-1},g_0x}(\{\exp(tY):t\geq 0\}) \text{ is HAW on }g_0Hg_0^{-1}.
\end{equation}

Let $\theta_i$ and $\theta$ be as in Proposition \ref{P:rank-one-power}, and renormalize $Y$ so that $1=\theta(Y)=\max_i\theta_i(Y).$ Take $\cO=\{Y\}$. Then the conclusion follows from Theorem \ref{T:para-winning}.
\end{proof}

\begin{theorem}\label{T:alg-prod}
    For $G,\Gamma$ as above, let $\cS$ denote the set of unit directions in $\fa^+$, and let $H$ be their common expanding horospherical subgroup of $G$. Then for any $x\in X=G/\Gamma$, the set
\begin{equation}\label{E:product-SH}
 \Pi_x:=\{(Y,h)\in \cS\times H:\{\exp(tY)hx:t\ge0\}\text{ is bounded in }X\}
\end{equation}
is HAW on $\cS\times H$.
\end{theorem}
\begin{proof}
First, we remove from $\cS$ the union $\cS_0$ of finitely many hyperplanes, on which either two highest roots take the same values, or a nonzero exterior weight vanishes. In view of Lemma \ref{L:HAW-mnfd} (iv), it suffices to prove the HAW property of the set \eqref{E:product-SH} on each connected component of $(\cS\setminus\cS_0)\times H$. Let $\cO_0$ be one of the connected components of $\cS\setminus\cS_0$. Then on the parameter space $\cO_0$, one highest root \(\theta_i\) is maximal, and all exterior weights have constant signs.

Consider the normalization map $Y\mapsto Y/\theta_i(Y)$ on $\cO_0$: its image $\cO$ is a relatively open subset of the affine subspace $\{\theta_i=1\}$ in $\fa$ satisfying all requirements in the beginning of Section \ref{S:global-local-estimates}. Then Theorem \ref{T:para-winning} implies that for any $x\in X$, the set $$
\{(Y,h)\in \cO\times H:\{\exp(tY)hx:t\ge0\}\text{ is bounded in }X\}
$$
is HAW on $\cO\times H$. Pulling back to $\cO_0\times H$ gives the desired conclusion.
\end{proof}

\subsection{Reductions in the semisimple setting} In this section, we assume that $G$ is a connected semisimple Lie group with finite center and without compact factors, and $\Gamma\leq G$ a lattice. Choose a maximal $\RR$-split Cartan subspace $\fa$, an open positive Weyl chamber $\fa^+$, and an Euclidean norm $\|\cdot\|$ on $\fa$. Let $\Sigma^+$ be the positive restricted root system, put $$
\cS=\left\{Y\in \fa^+:\|Y\|=1\right\},\qquad H=\exp\left(
\bigoplus_{\alpha\in\Sigma^+}\fg_{\alpha}
\right).
$$
Then $H$ is the expanding horospherical subgroup for any $Y\in\cS$.

We present the following result as an intermediate step towards the general theorems, which is of independent interest and can be reduced to Theorems \ref{T:alg-H} and \ref{T:alg-prod} by standard reduction.

\begin{theorem}\label{T:semisimple}
    Let $G,\Gamma$ be as above. Then \begin{enumerate}
        \item[(i)] for any one-parameter $\Ad$-$\RR$-diagonalizable subgroup $F$ of $G$ and any $x\in X=G/\Gamma$, the set $E_{H(F^+),x}(F^+)$ is HAW on $H(F^+)$; 
        \item[(ii)] for any $x\in X=G/\Gamma$, the set $\Pi_x$ is HAW on $\cS\times H$.
    \end{enumerate}
\end{theorem}

To prove Theorem \ref{T:semisimple}, we first recall some well-known structural results in the theory of semisimple Lie groups. Readers may refer to \cite{Morris15, Rag72} for more details.

\begin{lemma}\label{L:red}
    Let $G$ be a connected semisimple Lie group with finite center and without compact factors, and $\Gamma\leq G$ a lattice. Then there is a finite covering induced by the adjoint representation of $G$: $$
    Z(G)/(Z(G)\cap\Gamma)\to G/\Gamma\xrightarrow{\Ad}\Ad(G)/\Ad(\Gamma). 
    $$
    Moreover, there is a direct-product decomposition $\Ad(G)=G_1\times \cdots\times G_s$ such that $\Ad(\Gamma)$ is commensurable with $\Gamma_1\times \cdots\times \Gamma_s$, where $\Gamma_i:=\Ad(\Gamma)\cap G_i$ is an irreducible lattice in $G_i$.
\end{lemma}
\begin{proof}
    See \cite[Corollary 5.17]{Rag72} and \cite[Proposition 4.3.3]{Morris15}.
\end{proof}

After this standard reduction, we turn to the proof of Theorem \ref{T:semisimple}.

\begin{proof}[Proof of Theorem \ref{T:semisimple}]
    Note that all conclusions are stable under finite cover and products. By Lemma \ref{L:red}, we may assume that $G$ is a connected semisimple Lie group with trivial center and without compact factors, and that $\Gamma$ is an irreducible lattice in $G$. If $\Gamma$ is cocompact, then all conclusions are clear, so we may assume that $\Gamma$ is not cocompact.

    When $G$ has real rank $1$, we see that $\cS$ consists of a single point, and both conclusions (i) and (ii) follow from an essential modification of Dani's arguments in \cite[Proof of Theorem 4.1]{Dani86} (see Appendix \ref{A:Dani} for details). When $G$ has real rank at least $2$, it follows from Margulis' arithmeticity theorem (see \cite[Proposition 3.1.6 and Corollary 6.1.10]{Zimmer84}) that there exists a connected $\QQ$-simple algebraic group $\bG$ and a Lie group isomorphism
$\phi:\mathbf G(\mathbb R)^\circ\to G$
such that $\Gamma$ is commensurable with $\phi(\Lambda)$, where $\Lambda$ is an arithmetic subgroup of $\bG(\mathbb R)^\circ$. Then the conclusions (i) and (ii) follow from Theorems \ref{T:alg-H} and \ref{T:alg-prod} respectively.
\end{proof}

\subsection{Proof of Theorems \ref{T:main-homo} and \ref{T:main-H} on bounded orbits}
In this section, we prove Theorems \ref{T:main-homo} and \ref{T:main-H} on the winning property of bounded orbits on homogeneous spaces. 

We first deduce Theorem \ref{T:main-homo} from Theorem \ref{T:main-H} using the standard stable--neutral--unstable decomposition of the Lie algebra with respect to an $\Ad$-diagonalizable flow.

\begin{proof}[Proof of Theorem \ref{T:main-homo} assuming the truth of Theorem \ref{T:main-H}.] \text{Let
$F=\{g_t=\exp(tY):t\in\RR\}$} be a one-parameter $\Ad$-diagonalizable subgroup of $G$, let $F^+=\{g_t:t\geq 0\}$, and let $H=H(F^+)$ be its expanding horospherical subgroup. We first show that: if $E_{H,x}(F^+)$ is HAW on $H$ for any $x\in X$, then $E_X(F^+)$ is HAW on $X$.


Write $\fg=\Lie(G)$ and $D=\ad Y$. The assumption on $F$
implies that the complexification of $D$ is diagonalizable.
For each eigenvalue $\lambda$, let $\fg_\lambda$ denote its
eigenspace in $\fg_{\CC}$. Set
$$
 \fg^-=\fg\cap
       \bigoplus_{\Rep\lambda<0}\fg_\lambda,\qquad
 \fg^0=\fg\cap
       \bigoplus_{\Rep\lambda=0}\fg_\lambda,\qquad
 \fg^+=\fg\cap
       \bigoplus_{\Rep\lambda>0}\fg_\lambda.
$$
Note that $\fg=\fg^-\oplus\fg^0\oplus\fg^+$. Put $\fp=\fg^-\oplus\fg^0$ and $\fh=\fg^+$. It is clear that $\fh=\Lie(H)$. Moreover, for any $Z\in \fp$, 
\begin{equation}\label{hom:nonpositive-invariance}
 \exp(Z)y\in E_X(F^+)
 \Longleftrightarrow y\in E_X(F^+).
\end{equation}

Fix $x\in X$. Since the map $$
 \Phi_x:\fp\times H\to X,\qquad
 \Phi_x(Z,h)=\exp(Z)hx,
$$
has an invertible differential at $(0,1_G)$, by the inverse function theorem, there are open neighborhoods $U\subset\fp$ of $0$ and $V\subset H$ of $1_G$ such that $\Phi_x$ restricts to a diffeomorphism from $U\times V$ onto an open neighborhood $O_x$ of $x$.

Using \eqref{hom:nonpositive-invariance} with $y=hx$, we obtain
the exact identity
\[
 \Phi_x^{-1}\bigl(E_X(F^+)\cap O_x\bigr)
   =U\times\bigl(E_{H,x}(F^+)\cap V\bigr).
\]
Since $E_{H,x}(F^+)\cap V$ is HAW on $V$, by Lemma \ref{L:HAW-mnfd} (v), the above right-hand side is HAW on $U\times V$. Then Lemma \ref{L:HAW-mnfd} (iii) shows that $E_X(F^+)\cap O_x$ is HAW on $O_x$. Since $x\in X$ is arbitrary, this proves that $E_X(F^+)$ is HAW on $X$. 

Finally, the same arguments also show that: for $H^-:=H(F^-)$, if $E_{H^-,x}(F^-)$ is HAW on $H^-$ for any $x\in X$, then $E_X(F^-)$ is HAW on $X$. Applying Theorem \ref{T:main-H} to $F^+$ and $F^-$ gives the HAW property of $E_{H,x}(F^+)$ and $E_{H^-,x}(F^-)$ on $H^{\pm}$ for any $x\in X$. Therefore, we conclude that $E_X(F^+)$ and $E_X(F^-)$ are HAW on $X$, so is $E_X(F)=E_X(F^+)\cap E_X(F^-)$ by Lemma \ref{L:HAW-mnfd} (ii). This proves Theorem \ref{T:main-homo}.
\end{proof}

Now we prove Theorem \ref{T:main-H} by reducing the general case to Theorem \ref{T:semisimple} (i).

\begin{proof}[Proof of Theorem \ref{T:main-H}]
    We first reduce the proof to the case when $G$ is connected. In fact, write $x=g\Gamma\in X$. The connected component of $X$ containing $x$ is the image of the embedding $$
    G^0/(G^0\cap g\Gamma g^{-1})\to X,\qquad u\mapsto ux.
    $$
    Since this embedding is $G^0$-equivariant and $F,H$ both lie in $G^0$, we may instead work on this connected component of $X$ containing $x$. 

    We next reduce the proof to the case when $G$ is connected semisimple with trivial center and without compact factors. In fact, let $J$ denote the maximal connected closed amenable normal subgroup of $G$, and consider the projection induced by the group homomorphism $\pi:G\to G/J$: $$
    J/(J\cap\Gamma)\to G/\Gamma\xrightarrow{\pi}\pi(G)/\pi(\Gamma).
    $$
    Here $J\cap \Gamma$ is a cocompact lattice in $J$, $\pi(G)$ is a connected semisimple Lie group without compact factors, and $\pi(\Gamma)$ is a lattice in $\pi(G)$. Then for any $x\in X$, $$
    E_{H,x}(F^+)=H\cap \pi^{-1}(E_{\pi(H),\pi(x)}(\pi(F^+))).
    $$
    By applying Lemma \ref{L:HAW-mnfd} (v) locally, we see that the HAW property of $E_{H,x}(F^+)$ will follow from the HAW property of $E_{\pi(H),\pi(x)}(\pi(F^+))$. So we may instead work on the quotient space $\pi(G)/\pi(\Gamma)$ with the quotient flow $\pi(F)$. Furthermore, consider the finite cover induced by the adjoint representation of $\overline{G}:=\pi(G)$ and $\overline{\Gamma}:=\pi(\Gamma)$: $$
    Z(\overline{G})/(Z(\overline{G})\cap \overline{\Gamma})\to \overline{G}/\overline{\Gamma}\xrightarrow{\Ad}\Ad(\overline{G})/\Ad(\overline{\Gamma}).
    $$
    Here $Z(\overline{G})\cap \overline{\Gamma}$ is of finite index in $Z(\overline{G})$, $\Ad(\overline{G})$ is a connected semisimple Lie group with trivial center and without compact factors, and $\Ad(\overline{\Gamma})$ is a lattice in $\Ad(\overline{G})$.

    We finally reduce an $\Ad$-diagonalizable flow to its $\Ad$-$\RR$-diagonalizable part. Let $F^+=\{g_t:t\geq 0\}$ be a one-parameter $\Ad$-diagonalizable subsemigroup of $G$, and write its multiplicative Jordan decomposition as $g_t=k_tb_t=b_tk_t,$ where $\{k_t:t\geq 0\}$ has compact closure and $\{b_t:t\geq 0\}$ is $\Ad$-$\RR$-diagonalizable. It follows that $$
    H=H(F^+)=H(\{b_t:t\geq 0\}),\quad \text{and}\quad E_{H,x}(F^+)=E_{H,x}(\{b_t:t\geq 0\}).
    $$
    So we may instead work with the $\Ad$-$\RR$-diagonalizable flow $\{b_t:t\geq 0\}$.

    After these three reductions, we conclude that the conclusion follows from Theorem \ref{T:semisimple} (i).
\end{proof}

\subsection{Proof of Theorem \ref{T:main-bi} on bounded geodesics} In this section, we prove Theorem \ref{T:main-bi} on the winning property of bounded geodesics on locally symmetric spaces, by reducing it to Theorem \ref{T:semisimple} (ii).

We start with some basic structural results in the theory of semisimple Lie groups. Let $G,K,\Gamma,\fg,\fk,\fp$ be as in Section \ref{SS:geo}. Choose a maximal abelian subspace $\mathfrak a\subset\fp$ and an open positive chamber $\mathfrak a^+$. We set $\cS:=\fa^+\cap S(\fp)$, let $N^+$ (resp. $N^-$) be its expanding (resp. contracting) horospherical subgroup, and write $$
A=\exp(\fa),\qquad L=Z_{K}(\fa),\qquad P^-=LAN^-.
$$
Here $P^-$ is the minimal parabolic subgroup opposite to $N^+$.

\begin{lemma}\label{L:str}\begin{enumerate}
    \item Let $S(\fp)_{\mathrm{reg}}$ be the subset of vectors in $S(\fp)$ whose $\Ad(K)$-orbits have maximal dimension. Then \begin{equation}\label{geo:polar-coordinates}
 \Theta:\cS\times L\backslash K\to
                S(\fp)_{\mathrm{reg}},\qquad
 \Theta(v,Lk)=\Ad(k^{-1})v
\end{equation}
is a smooth diffeomorphism.
    \item Let $g\in G$. Then there is an open dense set
$Q\subset L\backslash K$ and a diffeomorphism $u:Q\to N^+$
such that for any $Lk\in Q$, one has
\begin{equation}\label{geo:big-cell-factorization}
 kg=p(k)u(Lk)\quad\text{ for some } p(k)\in P^-.
\end{equation}
The complement of $Q$ is a finite union of smooth submanifolds of
strictly smaller dimension.
\end{enumerate}
\end{lemma}
\begin{proof}
(1) It is clear that $\Theta$ is a well-defined smooth bijection. At any point $(v,Lk)$ in $\cS\times L\backslash K$, we calculate that $$
\mathrm{d}\Theta_{(v,Lk)}:(w,\overline{z})\mapsto \Ad(k^{-1})(w+[v,z]).
$$ 
Since this differential is always invertible, we see that $\Theta$ is a local diffeomorphism.~\\
(2) The Iwasawa decomposition $G=P^-K$, together with $P^-\cap K=L$, gives a diffeomorphism $L\backslash K\to P^-\backslash G$.
Let $g\in G$. Then the map
$$
 \sigma:L\backslash K\to P^-\backslash G,\qquad Lk\mapsto P^-kg
$$
is also a diffeomorphism. Consider an open dense subset of $P^-\backslash G$: $$
\cO=\{P^-u:u\in N^+\}.
$$
Then $Q:=\sigma^{-1}(\cO)$ is an open dense subset of $L\backslash K$, and its complement is a finite union of smooth submanifolds of
strictly smaller dimension. Since the map $$\varphi:N^+\to\mathcal O,\qquad u\mapsto P^-u$$ is a diffeomorphism, we see that the composite map $u:=(\varphi^{-1}\circ \sigma)\big|_Q$ is also a diffeomorphism such that \eqref{geo:big-cell-factorization} holds. This completes the proof.
\end{proof}

Now we prove Theorem~\ref{T:main-bi} by reducing it to Theorem \ref{T:semisimple} (ii).

\begin{proof}[Proof of Theorem~\ref{T:main-bi}]
Write $M=K\backslash G/\Gamma$, where $G$ is a connected semisimple Lie group with trivial center and without compact factors, $K\leq G$ is a maximal compact subgroup, and $\Gamma\leq G$ is a torsion-free lattice. We fix $p=[Kg\Gamma]\in M$, identify $S_pM$ with $S(\fp)$, and hence get 
\begin{equation*}\label{E:geo}
			E^+(p,\infty)\cong\{v\in S(\fp): \{K\exp(tv)g\Gamma: t\geq 0\} \text{ is bounded in }M\}=:E^+.
	\end{equation*}
Let $\Theta,u$ be given by Lemma \ref{L:str}, and write $x_0=1_G\Gamma\in X=G/\Gamma$. We claim that: \begin{equation}\label{E:keyid}
    (\cS\times Q)\cap \Theta^{-1}(E^+)=(\id_{\cS}\times u)^{-1}(\Pi_{x_0}).
\end{equation}

In fact, it is straightforward to check that: for any $(v,Lk)\in \cS\times Q$, \begin{subequations}
    \begin{align*}
        (v,Lk)\in\Theta^{-1}(E^+)&\Longleftrightarrow \{\exp(tv)kg\Gamma:t\geq 0\}\text{ is bounded in }X;\\
        (v,Lk)\in (\id_{\cS}\times u)^{-1}(\Pi_{x_0})&\Longleftrightarrow \{\exp(tv)u(Lk)\Gamma:t\geq 0\}\text{ is bounded in }X.
    \end{align*}
\end{subequations}
Note that $$
\exp(tv)kg\Gamma=(\exp(tv)p(k)\exp(-tv))\cdot \exp(tv)u(Lk)\Gamma,
$$
where $\{\exp(tv)p(k)\exp(-tv):t\geq 0\}$ is bounded in $G$. So we conclude that \eqref{E:keyid} holds.

Now Theorem~\ref{T:semisimple} (ii) implies that the set $\Pi_{x_0}$ is HAW on $\cS\times N^+$. By the claim \eqref{E:keyid} and Lemma \ref{L:HAW-mnfd} (iii), we see that the set $E^+$ is HAW on $\Theta(\cS\times Q)$. Moreover, note that
\begin{subequations}
    \begin{align*}
        S(\fp)_{\mathrm{reg}}\setminus\Theta(\cS\times Q)
       &=\Theta\bigl(\cS\times((L\backslash K)\setminus Q)\bigr);\\
       S(\fp)\setminus S(\fp)_{\mathrm{reg}}&=\{v\in S(\fp):\dim\Lie(Z_K(v))>\dim\Lie(L)\}
    \end{align*}
\end{subequations}
are both finite unions of smooth submanifolds of strictly smaller dimension. Then Lemma~\ref{L:HAW-mnfd} (iv) shows that the set $E^+$ is HAW on $S(\fp)$. This completes the proof.
\end{proof}

\appendix 
\section{Proof of Theorem \ref{T:semisimple} in the real rank 1 case}\label{A:Dani} In this appendix, we give a complete and detailed proof of Theorem \ref{T:semisimple} in the case when $G$ has real rank $1$. This result has been announced in \cite[Introduction]{KM96} and \cite[Section 7]{AGK} and was attributed to \cite{Dani86}; however, neither of these papers extracts a rigorous proof from Dani's arguments in \cite[Proof of Theorem 4.1]{Dani86}, especially in the case when the expanding horospherical subgroup $H$ of $G$ is non-abelian. We address this issue here.

\begin{proposition}\label{P:rank-one-haw}
Let $G$ be a connected semisimple real Lie group of real rank one, with finite center and without compact factors.  Let $\Gamma<G$ be a lattice, $F\leq G$ a one-parameter $\Ad$-$\RR$-diagonalizable subgroup of $G$, and write $H=H(F^+)$. Then for any $x\in X=G/\Gamma$, the set 
$$
 E_{H,x}=\{h\in H:F^+hx\text{ is bounded in }X\}
$$
is HAW on $H$.
\end{proposition}

\subsection{A winning lemma}

We begin with a simple lemma concerning the HAW property for certain subsets in Euclidean spaces. This is an adaptation of \cite[Theorem~3.2]{Dani86} to the hyperplane absolute game, with the same thickness and separation hypotheses.

For a nonempty subset $\cD\subset\RR^a$, we define its \textsl{thickness} as
$$
 \tau(\cD):=\inf_V\sup_{x,y\in \cD}\dist(x-y,V),
$$
where $V$ ranges over all linear hyperplanes in $\RR^a$. Write
$\tau(\emptyset)=0$. For an affine hyperplane $L$, write
$L^{(r)}=\{v:\dist(v,L)\leq r\}$. Then we note that if $\tau(\cD)<r$, then $\cD\subset L^{(r)}$ for some affine hyperplane $L$ in $\RR^a$.

\begin{lemma}\label{L:thin-separated-haw}
Let $\{\cD(p,\delta):p\in\NN,\delta\in (0,1)\}$ be a family of subsets in $\RR^a$. Suppose that for any compact subset $B$ of $\RR^a$ and any $\beta\in(0,1)$, there exist some $M\geq1$, $\epsilon\in(0,1)$, and a sequence $\{\tau_p:p\in\NN\}$ of positive numbers such that:
\begin{enumerate}
 \item[(i)] for any $p\in\NN$ and any $\delta\in(0,\epsilon)$ with
       $\cD(p,\delta)\cap B\neq\varnothing$, one has
       $$\tau_p\leq M\qquad\text{and}\qquad
        \tau(\cD(p,\delta))\leq \delta\tau_p;$$
 \item[(ii)] for any $p,p'\in\NN$ and any $\delta\in(0,\epsilon)$ with $\cD(p,\delta)\cap B\neq \varnothing$, $\cD(p',\delta)\cap B\neq\varnothing$, and $\beta\tau_p\leq\tau_{p'}\leq\beta^{-1}\tau_p$, one has either $p=p'$, or      $$\dist(\cD(p,\delta),\cD(p',\delta))\geq\epsilon(\tau_p+\tau_{p'}).$$
\end{enumerate}
Let $$
S=\bigcup_{0<\delta<1}\left(\RR^a\setminus\bigcup_{p\in\NN}\cD(p,\delta)\right).
$$
Then $S$ is hyperplane absolute winning.
\end{lemma}

\begin{proof}
Fix $0<\beta<1/3$, and let $B_0$ be Bob's initial closed ball with radius $\rho_0>0$. In view of our assumption applied to $B=B_0$, there exist some $M\geq 1$, $\epsilon\in (0,1)$ and a sequence $\{\tau_p:p\in\NN\}$ of positive numbers with the conditions (i) and (ii).
Choose
\begin{equation}\label{app:delta-choice}
 0<\kappa<\min\left\{\epsilon,\frac{\rho_0}{M}\right\},\qquad 
 0<\delta<\min\{\epsilon,\beta^2\kappa\}.
\end{equation}
By (i), each $\cD(p,\delta)$ meeting $B_0$ has $\tau_p\leq M<\rho_0/\kappa$.

Alice's move at $q=0$ may be arbitrary. Suppose Bob's balls are $B_q$ with radius $\rho_q>0$. Then $B_q\subseteq B_0$ and $\rho_q\geq\beta \rho_{q-1}$ for all $q\in\NN$.  If $\rho_q\not\to 0$ as $q\to +\infty$, then Alice wins by default.  Otherwise, let $z_\infty$ denote the unique point in $\bigcap_{q\in\NN}B_q$.  

For each $q\in\NN$, Alice considers the set of indices
\begin{equation}\label{app:scale-band}
 I_q:=\left\{p\in\NN:\frac{\rho_q}{\kappa}\leq\tau_p<\frac{\rho_{q-1}}{\kappa}\right\}.
\end{equation}
We claim that: $$
\left|\left\{p\in I_q:\cD(p,\delta)\cap B_q\neq\varnothing\right\}\right|\leq 1.
$$ In fact, for any $p,p'\in I_q$, we have $\beta\tau_p<\beta\rho_{q-1}/\kappa\leq \rho_q/\kappa\leq \tau_{p'}$, and similarly, $\tau_{p'}\leq \beta^{-1}\tau_p$. If $p\neq p'$ satisfy $\cD(p,\delta)\cap B_q\neq\varnothing$ and $\cD(p',\delta)\cap B_q\neq\varnothing$, then the condition (ii) gives 
$$
 \dist(\cD(p,\delta),\cD(p',\delta))\geq\epsilon(\tau_p+\tau_{p'})
 \geq\frac{2\epsilon}{\kappa}\rho_q>2\rho_q,
$$
which is impossible for two sets $\cD(p,\delta)$ and $\cD(p',\delta)$ meeting the same ball $B_q$ of radius $\rho_q$. This verifies the claim.

Let $q\in\NN$. If there exists an index $p\in I_q$ such that
$\cD(p,\delta)\cap B_q\neq\varnothing$, then this index
is unique, and by the condition (i), we have $$
 \tau(\cD(p,\delta))\leq\delta\tau_p<\frac{\delta}{\kappa}\rho_{q-1}<\beta^2\rho_{q-1}\leq \beta\rho_q.
$$
So Alice can choose an affine hyperplane $L_q\subseteq \RR^a$ such that $\cD(p,\delta)\subseteq L_q^{(\beta \rho_q)}$. This gives $$
\bigcup_{p\in I_q}(\cD(p,\delta)\cap B_q)\subseteq L_q^{(\beta \rho_q)}.
$$
Otherwise, Alice makes an arbitrary move.

Note that $\bigcup_{q\in\NN}I_q\supseteq\{p\in\NN:\cD(p,\delta)\cap B_0\neq \varnothing\}$. Then we have $$
z_{\infty}\in
\bigcap_{q\in\NN}B_q\setminus \bigcup_{q\in\NN}L_q^{(\beta \rho_q)}
\subseteq 
\bigcap_{q\in\NN}B_q\setminus \bigcup_{q\in\NN}\bigcup_{p\in I_q}\cD(p,\delta)
\subseteq B_0\setminus \bigcup_{p\in\NN}\cD(p,\delta),
$$
so $z_{\infty}\in S$. This proves that the set $S$ is HAW on $\RR^a$.
\end{proof}

\subsection{Notations from Dani's paper}\label{SS:Dani-notation}

We collect some notations from \cite[Sections~1, 2 and 4]{Dani86} that are used later. In this subsection, we assume that $\Gamma$ is not cocompact in $G$.

\subsubsection{Cusp neighborhoods for $G/\Gamma$}
Let $G$ be a connected semisimple real Lie group of real rank one, with finite center and without compact factors, and let $\Gamma\leq G$ be a lattice. We write $A=F=\{\exp(tY):t\in\RR\}$, and $$
 A^+=\{\exp(tY):t>0\},\qquad A^-=\{\exp(tY):t<0\},\qquad A_s=\{\exp(-tY):t>s\}.
$$ 
Write $N=H(F^+)$ and $N^-=H(F^-)$. Let $K\leq G$ be a maximal compact subgroup compatible with $A$, and write $L=Z_{K}(A)$, $P=N_G(N)=LAN$. 

Choose $w\in K$ such that $waw^{-1}=a^{-1}$ for all $a\in A$. Then $wNw^{-1}=N^-$.
We have the Iwasawa decomposition and the Bruhat decomposition
$$ G=KAN,\qquad G=PwN\sqcup P,$$
and write the Iwasawa projection to be $$
 \pi:G\to N,\qquad \pi(kan)=n
 \quad\text{for }k\in K,\ a\in A,\ n\in N.
$$
By \cite[Proposition~2.2]{Dani86}, there exists a compact set $\Sigma\subset N$, some $\rho\in\RR$, and a finite set $\Lambda\subset G$ of cusp representatives such that
$$G=KA_{\rho}\Sigma\Lambda\Gamma.$$
We also choose $\sigma\geq\rho$ as in \cite[Remark~2.3]{Dani86} so that two cusp neighborhoods $KA_\sigma N\lambda\gamma$ and $KA_\sigma N\lambda'\gamma'$ are equal if and only if
$$
 \lambda=\lambda'\quad\text{and}\quad
 \lambda\gamma\gamma'^{-1}\lambda^{-1}\in LN.
$$
Otherwise they are disjoint.

Fix $\lambda\in\Lambda$. We choose a sequence of representatives $(g_p)_{p\in\NN}$ in $\Lambda\Gamma\lambda^{-1}\setminus P$ such that $$
\{(LN)g:g\in \Lambda\Gamma\lambda^{-1}\setminus P\}
=\{(LN)g_p:p\in\NN\},\quad \text{where }(LN)g_p \text{ are pairwise disjoint}.
$$
For $p\in\NN$, we write the Bruhat decomposition of $g_p$ as
\begin{equation}\label{app:bruhat-data}
 g_p=m_pa_pw n_p,\quad\text{where }
 m_p\in LN,\quad a_p=\exp(t_pY)\in A,\quad n_p\in N.
\end{equation}
Then the sets $KA_\sigma Ng_p$ are pairwise disjoint and are disjoint from $KA_\sigma N$. 
The same argument as in \cite[Proof of Proposition~2.5]{Dani86} gives
\begin{equation}\label{app:height-bound}
 t_p\leq2\sigma\qquad\text{for all }p\in\NN.
\end{equation}
Write $\fn=\Lie(N)$. 
We choose an inner product on $\fn$ for which the eigenspaces of $\Ad(\exp Y)|_{\mathfrak n}$ are orthogonal. Let $\eta>1$ be its largest eigenvalue.

\subsubsection{A sequence $\{\tau_p:p\in\NN\}$ and a family of dangerous sets $\{\cD(p,\delta):p\in\NN,\;\delta\in (0,1)\}$}

By \cite[Corollary~1.7 and Lemma~1.8]{Dani86}, the set $\pi(A^+N^-)$ is relatively compact and contains a neighborhood of the identity. Write $\Omega=\log\pi(A^+N^-)$, and choose 
\begin{equation}\label{E:tau-p}
 C:=\max\left\{1,2\sup_{z\in\Omega}\|z\|\right\},\qquad M:=\max\{1,C\eta^{2\sigma}\},\qquad \tau_p:=C\eta^{t_p}.
\end{equation}
Then $0<\tau_p\leq M$ for all $p\in\NN$.

 For $\delta>0$, write
$$
 b_\delta=\exp\left(-\frac{\log\delta}{\log\eta}Y\right),
 \qquad
 D_\delta=\Ad(b_\delta^{-1})|_{\mathfrak n}.
$$
Note that 
$D_\delta\to0$ as $\delta\to 0^+$.

For $p\in\NN$ and $0<\delta<1$, we define the dangerous set $\cD(p,\delta)$ to be \begin{equation}\label{E:Dani-danger}
    \cD(p,\delta):=\log\pi(A^+N^-b_{\delta}a_p^{-1}n_p)=\log(\exp(D_{\delta \tau_p/C}(\Omega))\cdot n_p).
\end{equation}
Then our target set is \begin{equation}\label{E:S-target}
    S:=\bigcup_{0<\delta<1}
 \left(\fn\setminus\bigcup_{p\in\NN}\cD(p,\delta)\right).
\end{equation}
The same argument as in \cite[Proposition~2.5]{Dani86} gives $S=\log{E_{N,\lambda\Gamma}}.$

\subsection{Proof of Proposition \ref{P:rank-one-haw}} In this section, we give a rigorous proof of Proposition \ref{P:rank-one-haw}. If the lattice $\Gamma$ is cocompact in $G$, or $F$ is trivial, then the conclusion clearly holds. So we may assume that $\Gamma$ is not cocompact in $G$, and that $F=\{\exp(tY):t\in\RR\}$ for some nonzero $\ad$-$\RR$-diagonalizable $Y$. Following notations in Section \ref{SS:Dani-notation}, we claim that: the set $S$ as in \eqref{E:S-target} is HAW on $\fn$. 

Assume the truth of this claim for a moment. Since $\exp:\fn\to N$ is a diffeomorphism, by Lemma \ref{L:HAW-mnfd} (iii), we see that the set $E_{N,\lambda\Gamma}$ is HAW on $N$. This proves Proposition \ref{P:rank-one-haw} for any $x=\lambda\Gamma$ where $\lambda\in\Lambda$ is any fixed cusp representative for $G/\Gamma$. 

In general, for any $g\in G$, we consider the diffeomorphism (using Bruhat decomposition) $$
\psi_g:N\to (P\backslash G)\setminus \{Pg\},\qquad n\mapsto Pwng.
$$ 
It is clear that \begin{equation}\label{E:P-G-bounded}
\psi_g\left(E_{N,g\Gamma}\right)=\{Ph\in P\backslash G:F^-h\Gamma\text{ is bounded in }X\}.
\end{equation}
It follows from Lemma \ref{L:HAW-mnfd} (iii) and (iv) that the set $E_{N,g\Gamma}$ is HAW on $N$ if and only if the right-hand side of \eqref{E:P-G-bounded} is HAW on $P\backslash G$. Note that the right-hand side of \eqref{E:P-G-bounded} is independent of $g\in G$. So we conclude that for any two $g_1,g_2\in G$, $$
E_{N,g_1\Gamma}\text{ is HAW on }N\Longleftrightarrow E_{N,g_2\Gamma}\text{ is HAW on }N.
$$
In particular, when Proposition \ref{P:rank-one-haw} holds for one $x=\lambda\Gamma$ where $\lambda\in\Lambda$, it holds for all $x\in X$.

The remaining task is to verify the above claim. Recall that in Section \ref{SS:Dani-notation}, for a fixed $\lambda\in \Lambda$, we have defined a constant $M\geq 1$, a sequence $\{\tau_p:p\in\NN\}$ and a family of dangerous sets $\{\cD(p,\delta):p\in\NN,\;\delta\in (0,1)\}$, independent of choices of $B$ and $\beta$. In view of Lemma \ref{L:thin-separated-haw}, it suffices to verify that: given any compact subset $B\subseteq \fn$ and any $\beta\in (0,1)$, there exists some $\epsilon=\epsilon(B,\beta)>0$ such that the conditions (i) and (ii) hold.

\subsubsection{Verification of the condition (i)}\label{SS:thinness}
We first verify the condition (i) of Lemma~\ref{L:thin-separated-haw} for the dangerous sets $\cD(p,\delta)$ defined in \eqref{E:Dani-danger}. The bound $\tau_p\leq M$ is already given by \eqref{E:tau-p}.

If $N$ is abelian, the rank-one restricted-root decomposition gives $D_{\delta}=\delta\cdot\mathrm{id}_{\fn}$. Then
$$
 \cD(p,\delta)=\frac{\delta\tau_p}{C}\Omega+\log n_p,
$$
and hence
$$
 \tau(\cD(p,\delta))\leq\frac{\delta\tau_p}{C}\operatorname{diam}\Omega
 \leq\delta\tau_p.
$$

Now we assume that $N$ is non-abelian. By the rank-one restricted-root decomposition,
$$
 \fn=\fg_{\alpha}\oplus \fg_{2\alpha},\qquad
 \text{where }\alpha\text{ is the smallest positive restricted root associated with }A.
$$
Then we have $D_{\delta}=\sqrt{\delta}\cdot \id_{\fg_{\alpha}}\oplus \delta\cdot \id_{\fg_{2\alpha}}.$
Now we use the Baker--Campbell--Hausdorff formula to calculate $\cD(p,\delta)$. Write $n_p=\exp(u_p,v_p)$ and $\delta_p=\delta\tau_p/C$. Then \begin{equation}
    \begin{aligned}
        \cD(p,\delta)&=\left\{
        \log\left(
        \exp\left(\delta_p^{1/2}u_0,\delta_pv_0\right)\cdot \exp(u_p,v_p)
        \right):(u_0,v_0)\in\Omega
        \right\}\\
        &=\left\{
        \left(
        \delta_p^{1/2}u_0+u_p,\;
        \delta_p v_0+v_p+\frac{1}{2}\left[\delta_p^{1/2}u_0,u_p\right]
        \right):(u_0,v_0)\in\Omega
        \right\}.
    \end{aligned}
\end{equation}
This means that for any $(u,v)\in \cD(p,\delta)$, there exists some $(u_0,v_0)\in\Omega$ such that
$$
 u=\delta_p^{1/2}u_0+u_p,
 \qquad
 v=\delta_p v_0+v_p+\frac{1}{2}\left[\delta_p^{1/2}u_0,u_p\right].
$$
Fix any unit linear functional $\ell\in \fg_{2\alpha}^*$. It follows that
$$
 \left|\ell\left(v-v_p-\frac12[u-u_p,u_p]\right)\right|=\delta_p|\ell(v_0)|\leq\frac{\delta\tau_p}{2}.
$$
Thus, for the affine hyperplane in $\fn$:
$$
 L_p:=\left\{(u,v)\in\fg_{\alpha}\oplus\fg_{2\alpha}:\ell\left(v-v_p-\frac12[u-u_p,u_p]\right)=0\right\},
$$
we have
$\cD(p,\delta)\subset L_p^{(\delta\tau_p/2)}$, and hence $\tau(\cD(p,\delta))\leq\delta\tau_p.$ This verifies the condition (i).

\subsubsection{Verification of the condition (ii)}\label{SS:separation} Now we verify the condition (ii) of Lemma~\ref{L:thin-separated-haw}. Fix a compact subset $B\subset\mathfrak n$ and $\beta\in(0,1)$. The proof is divided into several steps.

\smallskip
\noindent\textit{Step 1. Separation of the normalized centers.}
Write $$c_p=\exp(\sigma Y)a_p^{-1}.$$ Following \cite[Proof of Theorem~4.1]{Dani86}, using cusp disjointness and \cite[Lemma~1.8]{Dani86}, we see that there exists a neighborhood $U$ of the identity in $N$ such that
\begin{equation}\label{app:center-separation}
 c_p(n_pn_{p'}^{-1})c_p^{-1}\notin U
\end{equation}
whenever $p\neq p'$ and $\beta\tau_p\leq\tau_{p'}\leq\beta^{-1}\tau_p$.

\smallskip
\noindent\textit{Step 2. Separation after normalization.}
We choose and fix a symmetric neighborhood $U_1$ of the identity such that $U_1^3\subseteq U$. Since $\Omega$ is bounded and $D_\delta\to0$ as $\delta\to0^+$, we can choose $\delta_0\in(0,1)$ sufficiently small such that
$$
 \exp(D_{\delta\eta^\sigma s}(\Omega))\subset U_1
 \qquad \text{for all }\delta\in (0,\delta_0),\;s\in [\beta,\beta^{-1}].
$$

Let $p\neq p'$ satisfy $\beta\tau_p\leq\tau_{p'}\leq\beta^{-1}\tau_p$, and let $\delta\in (0,\delta_0)$. For any $x\in\cD(p,\delta)$ and $x'\in\cD(p',\delta)$, by \eqref{E:Dani-danger}, there exist $z,z'\in\Omega$ such that
$$
 \exp x=\exp(D_{\delta\tau_p/C}(z))\cdot n_p,\qquad
 \exp x'=\exp(D_{\delta\tau_{p'}/C}(z'))\cdot n_{p'}.
$$
Using $\Ad(c_p)\big|_{\fn}=D_{C\eta^{\sigma}/\tau_p}$ and $\tau_p=C\eta^{t_p}$,
we obtain
$$
 \begin{aligned}
 c_p(n_pn_{p'}^{-1})c_p^{-1}
 &=
 \exp(-D_{\delta\eta^\sigma}(z))\cdot
 \bigl(c_p\exp x\exp(-x')c_p^{-1}\bigr)\cdot 
 \exp(D_{\delta\eta^{\sigma}\tau_{p'}/\tau_p}(z')).
 \end{aligned}
$$
Note that both the first and third factors on the right-hand side belong to $U_1$, but the left-hand side does not belong to $U_1^3$. It follows that the middle factor on the right-hand side does not belong to $U_1$. This gives
\begin{equation}\label{E:normalized-separation}
 \|\Ad(c_p)\log(\exp x\exp(-x'))\|\gg_\beta 1.
\end{equation}

\smallskip
\noindent\textit{Step 3. Separation in the Euclidean norm.}
Suppose in addition that $\cD(p,\delta)$ and $\cD(p',\delta)$ both meet $B$. Following \cite[Lemma~4.6 and (4.7)]{Dani86}, we see that they lie in a fixed compact subset of $\fn$, on which we have the uniform estimate
$$
 \|\log(\exp x\exp(-x'))\|\ll_B\|x-x'\|.
$$
Note that
$$
 \|\Ad(c_p)|_{\fn}\|=\|D_{C\eta^\sigma/\tau_p}\|
 \leq\max\left\{
 \sqrt{\frac{C\eta^\sigma}{\tau_p}},\;\frac{C\eta^\sigma}{\tau_p}
 \right\}
 \ll\frac{1}{\tau_p}.
$$
Combining these estimates with \eqref{E:normalized-separation} gives $\|x-x'\|\gg_{B,\beta}\tau_p.$ Since $\tau_{p'}\leq\beta^{-1}\tau_p$, taking the infimum over $x\in\cD(p,\delta)$ and $x'\in\cD(p',\delta)$ yields
$$
 \dist(\cD(p,\delta),\cD(p',\delta))
 \gg_{B,\beta}\tau_p+\tau_{p'}.
$$
Therefore, by choosing $\epsilon=\epsilon(B,\beta)\in (0,\delta_0)$ sufficiently small, we obtain
$$
 \dist(\cD(p,\delta),\cD(p',\delta))
 \geq\epsilon(\tau_p+\tau_{p'})
 \qquad\text{for all }0<\delta<\epsilon.
$$
This verifies the condition (ii). The proof is complete.

\section*{Acknowledgment}
The author would like to thank Lifan Guan for suggesting these problems and proposing a promising approach to solving them. After finishing the paper \cite{Wu26} with the assistance of AI tools, the author realized that the simplified arguments there can be adapted to prove the winning result (Theorem \ref{T:para-winning}) on the product of a parameter space and its corresponding expanding horospherical subgroup, from which all main theorems will follow.

The author would also like to express his deepest gratitude to his Ph.D. advisor Jinpeng An and his co-advisor Dmitry Kleinbock during his self-funded visit to Brandeis University, for their constant encouragement along the winding way.

\section*{Declaration of AI Use}
Based on the preceding AI-assisted work \cite{Wu26}, the author independently developed the basic approach to proving the main theorems in the current paper. During the preparation of this work, the author used GPT-6 Pro to assist in doing calculations and checking proofs in Section \ref{S:global-local-estimates}, as well as in finding the modified proof in Appendix \ref{A:Dani}. 

The author would like to emphasize here that he had long questioned whether a proof of Theorem \ref{T:semisimple} in the real rank $1$ case could indeed be extracted from the arguments in \cite[Proof of Theorem 4.1]{Dani86}, as stated in \cite[Introduction]{KM96} and \cite[Section 7]{AGK}. With the assistance of AI tools, he was eventually able to carry out this deduction rigorously.

The author has independently verified all mathematical arguments and takes full responsibility for the manuscript.


\begin{thebibliography}{99}

\bibitem{An13} J.\ An, \textsl{Badziahin-Pollington-Velani's theorem and Schmidt's game}, Bull. Lond. Math. Soc. {\bf 45} (2013), no. 4, 721--733.
			
\bibitem{An16} \bysame, \textsl{2-dimensional badly approximable vectors and Schmidt's game}, Duke Math. J. {\bf 165} (2016), no. 2, 267--284.

\bibitem{AGGL} J.\ An, A.\ Ghosh, L.\ Guan, T.\ Ly, \textsl{Bounded orbits of diagonalizable flows on finite volume quotients of products of $\SL_2(\RR)$}, Adv. Math. {\bf 354} (2019), 106743.

\bibitem{AGK} J.\ An, L.\ Guan, D.\ Kleinbock, \textsl{Bounded orbits of diagonalizable flows on $\SL_3(\RR)/\SL_3(\ZZ)$}, Int. Math. Res. Not. IMRN {\bf 2015} (2015), no. 24, 13623--13652. 

\bibitem{BH17} D. Badziahin and S. Harrap, \textsl{Cantor-winning sets and their applications}, Advances in Mathematics {\bf 318} (2017): 627--677.

            \bibitem{BHNS18} D. Badziahin, S. Harrap, E. Nesharim, and D. Simmons, \textsl{Schmidt Games and Cantor Winning Sets}, Ergodic Theory and Dynamical Systems {\bf 45}, no. 1 (2025): 71--110.

\bibitem{BM00} M. B. Bekka and M. Mayer, \textsl{Ergodic Theory and Topological Dynamics of Group Actions on Homogeneous Spaces}, London Mathematical Society Lecture Note Series, vol.~269, Cambridge University Press, Cambridge, 2000.

\bibitem{Benoist09} Y.\ Benoist, \textsl{Five lectures on lattices in semisimple Lie groups}, in G\'eom\'etries \`a courbure n\'egative ou nulle, groupes discrets et rigidit\'es, S\'emin. Congr. {\bf 18}, Soc. Math. France, Paris, 2009, 117--176.

\bibitem{BNY21} V.\ Beresnevich, E.\ Nesharim, L.\ Yang, \textsl{$\mathrm{Bad}(\mathbf{w})$ is hyperplane absolute winning}, Geom. Funct. Anal. {\bf 31} (2021), 1--33. 

\bibitem{BNY22} \bysame, \textsl{Winning property of badly approximable points on curves}, Duke Math. J. {\bf 171} (2022), no. 14, 2841--2880. 

\bibitem{BE95} P.~Borwein and T.~Erd\'elyi, \textsl{Polynomials and Polynomial Inequalities}, Graduate Texts in Mathematics, vol.~161, Springer-Verlag, New York, 1995.

\bibitem{BBFKW10} R.\ Broderick, Y.\ Bugeaud, L.\ Fishman, D.\ Kleinbock, B.\ Weiss, \textsl{Schmidt's game, fractals, and numbers normal to no base}, Math.\ Res.\ Lett., {\bf 17} (2010), no.\ 2, 307--321.
			
			
			\bibitem{BFK11} R.\ Broderick, L.\ Fishman, D.\ Kleinbock, \textsl{Schmidt's game, fractals, and orbits of toral endomorphisms},  Ergodic Theory Dynam.\ Systems {\bf 31}  (2011),  no.\ 4, 1095--1107.
            
            \bibitem{BFKRW12} R.\ Broderick, L.\ Fishman, D.\ Kleinbock, A.\ Reich, B.\ Weiss, \textsl{The set of badly approximable vectors is strongly $C^1$ incompressible}, Math.\ Proc.\ Cambridge Philos.\ Soc.\ {\bf 153} (2012), no.\ 2, 319--339.
			
			



\bibitem{Cassel} J. W. S. Cassels, \textsl{An introduction to the geometry of numbers}, Corrected reprint of the 1971 edition, Springer-Verlag, Berlin, 1997.

\bibitem{Dan79} S. G.\ Dani, \textsl{On invariant measures, minimal sets and a lemma of Margulis}, Invent.\ Math.\ {\bf 51} (1979), 239--260.

\bibitem{Dani85} \bysame, \textsl{Divergent trajectories of flows on homogeneous spaces and Diophantine approximation}, J. Reine Angew. Math. {\bf 359} (1985), 55--89.

\bibitem{Dani86} \bysame, \textsl{Bounded orbits of flows on homogeneous spaces}, Comment. Math. Helv. {\bf 61} (1986), 636--660. 

\bibitem{Dan86} \bysame, \textsl{On orbits of unipotent flows on homogeneous spaces, II}, Ergodic Theory Dynam.\ Systems {\bf 6} (1986), no.\ 2, 167--182.

\bibitem{EGL16} M.\ Einsiedler, A.\ Ghosh, B.\ Lytle, \textsl{Badly approximable vectors, $C^1$ curves and number fields}, Ergodic Theory Dynam. Systems {\bf 36} (2016), no. 6, 1851--1864.

\bibitem{GW26} L.\ Guan, C.\ Wu, \textsl{Bounded geodesics on locally symmetric spaces}, Ergodic Theory Dynam. Systems {\bf 46} (2026), no.~8, 1969--1993.

\bibitem{GW18} L.\ Guan, W.\ Wu, \textsl{Bounded orbits of certain diagonalizable flows on $\SL_n(\RR)/\SL_n(\ZZ)$}, Trans. Amer. Math. Soc. {\bf 370} (2018), no. 7, 4661--4681.

\bibitem{Hum72} J.\ E.\ Humphreys, \textsl{Introduction to Lie Algebras and Representation Theory}, Graduate Texts in Mathematics, vol.~9, Springer-Verlag, New York, 1972.

\bibitem{KLW04} D. Y. \ Kleinbock, E.\ Lindenstrauss, and B.\ Weiss, \textsl{On fractal measures and Diophantine approximation}, Selecta Math.\ (N.S.) {\bf 10} (2004), no.\ 4, 479--523.

\bibitem{KL16} D. Y. \ Kleinbock, T.\ Ly, \textsl{Badly approximable $S$-numbers and absolute Schmidt games}, J. Number Theory {\bf 164} (2016), 13--42.

\bibitem{KM96} D. Y.\ Kleinbock, G. A.\ Margulis, \textsl{Bounded orbits of nonquasiunipotent flows on homogeneous spaces}, Amer. Math. Soc. Transl. Ser. 2 {\bf 171} (1996), 141--172.

\bibitem{KM98} \bysame, \textsl{Flows on homogeneous spaces and Diophantine approximation on manifolds}, Ann.\ of Math.\ (2) {\bf 148} (1998), no.\ 1, 339--360.

\bibitem{KW13} D. Y.\ Kleinbock, B.\ Weiss, \textsl{Modified Schmidt games and a conjecture of Margulis}, J. Mod. Dyn. {\bf 7} (2013), no. 3, 429--460. 

\bibitem{KW15} \bysame,  \textsl{Values of binary quadratic forms at integer points and Schmidt games}, in: {\bf Recent trends in ergodic theory and dynamical systems (Vadodara, 2012)}, pp.\  77--92, Contemp.\ Math. {\bf 631}, Amer.\ Math.\ Soc., Providence, RI, 2015.

\bibitem{Knapp02} A.\ W.\ Knapp, \textsl{Lie Groups Beyond an Introduction}, 2nd ed., Progress in Mathematics, vol.~140, Birkh\"auser, Boston, 2002.

\bibitem{Mar71} G. A.\ Margulis, \textsl{On the action of unipotent groups in a lattice space}, Math.\ USSR-Sb.\ {\bf 15} (1971), no.\ 4, 549--554.

\bibitem{McM10} C. T. McMullen, \textsl{Winning sets, quasiconformal maps and Diophantine approximation}, Geom. Funct. Anal. \textbf{20} (2010), no.~3, 726--740.

\bibitem{Morris15} D.\ W. Morris, \textsl{Introduction to Arithmetic Groups}, Deductive Press, 2015.

\bibitem{Rag72} M. S. Raghunathan, \textsl{Discrete Subgroups of Lie Groups}, Ergebnisse der Mathematik und ihrer Grenzgebiete, vol.~68, Springer-Verlag, New York--Heidelberg, 1972.

\bibitem{Saw23} G.\ Sawant, \textsl{Weighted badly approximable complex vectors and bounded orbits of certain diagonalizable flows}, Int. J. Number Theory {\bf 19} (2023), no. 8, 1977--1993.

\bibitem{Saw25} \bysame, \textsl{Hyperplane absolute winning property of bounded orbits under diagonalizable flows on $\SL_3(\CC)/\SL_3(\mathcal{O}_K)$}, Proc. Indian Acad. Sci. Math. Sci. {\bf 135} (2025), Paper No. 7.


\bibitem{Sch66} W. M.\ Schmidt,  \textsl{On badly approximable numbers and certain games}, Trans.\ Amer.\ Math.\ Soc.\ {\bf 123} (1966), 178--199.

\bibitem{Sch69}    \bysame,  \textsl{Badly approximable systems of linear forms}, J.\ Number Theory {\bf 1} (1969), 139--154.

\bibitem{Waldschmidt92} M.\ Waldschmidt, \textsl{Linear independence of logarithms of algebraic numbers}, IMSc Report {\bf 116}, The Institute of Mathematical Sciences, Madras, 1992.

\bibitem{Wu26} C.\ Wu, \textsl{$\Bad(\mathbf{r};\mathbf{s})$ is hyperplane absolute winning}, preprint, 2026.

\bibitem{Zimmer84} R. J. Zimmer, \textsl{Ergodic Theory and Semisimple Groups}, Monographs in Mathematics {\bf 81}, Birkhäuser, 1984.

\end{thebibliography}
\end{document}